\documentclass[11pt,letterpaper]{amsart}
\usepackage[plainpages=false]{hyperref}
\usepackage{amsfonts,latexsym,rawfonts,amsmath,amssymb,amsthm, mathrsfs, lscape, calligra}
\usepackage{MnSymbol}
\usepackage{verbatim}
\usepackage{enumerate}
\usepackage{centernot}
\usepackage{mathtools}
\usepackage{tikz}
\usepackage{tikz-cd}
\usepackage{}
\usepackage[all]{xy}
\usepackage{graphicx,psfrag}

\usepackage{array, tabularx, color}

\usepackage{setspace}

\newtheorem{thm}{Theorem}[section]
\newtheorem{cor}[thm]{Corollary}
\newtheorem{lem}[thm]{Lemma}
\newtheorem{clm}[thm]{Claim}
\newtheorem{prop}[thm]{Proposition}

\newtheorem{exam}[thm]{Example}

\theoremstyle{remark}
\newtheorem{rmk}[thm]{Remark}

\theoremstyle{definition}
\newtheorem{defi}[thm]{Definition}

\newcommand{\CBbb}{\mathbb C}

\newcommand{\PBbb}{\mathbb P}

\newcommand{\RBbb}{\mathbb R}

\newcommand{\Ccal}{\mathcal C}

\newcommand{\Ecal}{\mathcal E}
\newcommand{\Fcal}{\mathcal F}
\newcommand{\Gcal}{\mathcal G}
\newcommand{\Hcal}{\mathcal H}
\newcommand{\Ical}{\mathcal I}

\newcommand{\Kcal}{\mathcal K}

\newcommand{\Ocal}{\mathcal O}
\newcommand{\Pcal}{\mathcal P}
\newcommand{\Qcal}{\mathcal Q}

\newcommand{\Tcal}{\mathcal T}

\DeclareMathOperator{\dvol}{dvol}

\DeclareMathOperator{\Img}{Im}

\numberwithin{equation}{section}

\makeatletter
\@namedef{subjclassname@2020}{\textup{2020} Mathematics Subject Classification}
\makeatother

\begin{document}
\title[Point Singularities and Bubbling]{Point Singularities and Bubbling in Degenerations of Rank-Two Bundles on Threefolds}
\author[Chen]{Xuemiao Chen}
\address{Department of Pure Mathematics, University of Waterloo, Waterloo,
Ontario, Canada N2L 3G1}
\email{x67chen@uwaterloo.ca}

\begin{abstract}
We study one-parameter degenerations of rank-two vector bundles on complex threefolds to a rank-two torsion-free sheaf with an isolated point singularity. We prove a rigidity identity: the algebraic bubbling multiplicity of the
central fiber equals one half of the Ext-length of the singularity of its
reflexive hull. Furthermore, bubbling is forced when a family develops such an isolated point singularity. We use this identity to obtain smoothability obstructions and
construct sharp local smoothings. We realize the local example from earlier joint work with Sun as a global
degeneration of smooth Hermitian--Yang--Mills connections. Rescaling this
degeneration produces a smooth non-flat Hermitian--Yang--Mills connection
on $\mathbb C^3$ with density one at infinity, whose tangent cone at
infinity has flat connection part and a multiplicity-one line as the blow-up cycle.
We also construct smoothings of elementary modifications of projective-cone
singularities with explicit algebraic bubbles. These examples give local models for Hermitian--Yang--Mills point
bubbling in complex dimension three and distinguish this phenomenon from
bubbling along complex codimension-two loci.
\end{abstract}
\maketitle

\tableofcontents

\thispagestyle{empty}

\bibliographystyle{amsplain}
\section{Introduction}
We study one-parameter degenerations whose general fibers are rank-two locally free sheaves on a complex threefold and whose central fiber is a rank-two torsion-free sheaf with an isolated point singularity. In particular, such families have to be flat over the parameter space (see Lemma \ref{ReflexiveProperty}). At first sight, one might expect
such a point defect to be merely a local singularity of the limiting
sheaf, independent of algebraic bubbling, meaning that the central fiber coincides
with its reflexive hull. This expectation is not
unreasonable: in higher rank, an isolated essential singularity can indeed
occur on its own, without forcing algebraic bubbling; see Example
\ref{BubbleNotSymmetric}. The main point of this paper is that rank two
behaves differently. For rank-two bundles, an isolated point singularity cannot form on its own:
it necessarily carries algebraic bubbling, and the bubbling multiplicity is
rigidly determined by the singularity of the reflexive hull. This rigidity leads to smoothability obstructions, sharp
local models, and new examples of point bubbling for Hermitian--Yang--Mills
connections on K\"ahler threefolds. In particular, one of the resulting
analytic bubbles is a smooth non-flat Hermitian--Yang--Mills connection on
$\mathbb C^3$ with minimal density at infinity: its tangent cone at infinity
has flat connection part, but carries a multiplicity-one complex line as
its blow-up cycle. The local problem is motivated by three related analytic considerations.

\begin{itemize}
\item The first motivation comes from recent work on tangent cones for
admissible Hermitian--Yang--Mills (HYM) connections. In joint work with Song Sun \cite{ChenSun:18, ChenSun:19, ChenSun:20a,
ChenSun:20b}, the analytic tangent cones were shown to be algebraic
invariants of the stalk of the underlying sheaf at the point. However, there are examples in which the original admissible HYM connection has a genuine
essential singularity at the point, but its analytic tangent-cone connection is flat. Thus the singularity is invisible at the level of the tangent-cone
connection. This shows that, in order to understand singularities of
admissible HYM connections, one cannot rely only on a Federer-type
dimension-reduction argument based on tangent cone connections. This motivates us to
study the process by which such singularities form.

\item The second motivation comes from the analytic subtleties in
gauge-theoretic approaches to Donaldson--Thomas invariants for
Calabi--Yau threefolds using stable vector bundles
\cite{DonaldsonThomas:98, Thomas:2000}. Two main technical issues are
transversality and compactness. The compactness issue is largely due to
analytic bubbling, which in this setting may occur both along curves and
at isolated points. Along a generic part of a bubbling curve, the local
model is expected to be a family of instantons on $\RBbb^4$
parameterized by the curve, and Fueter sections are expected to play a
key role in describing this model. By contrast, isolated point bubbling
has no such transverse four-dimensional model and remains much less
understood.

\item A third motivation comes from $G_2$-gauge theory. After taking the
product with $\RBbb$, isolated point bubbling for HYM connections on
$\CBbb^3$ gives a local model for bubbling along a real one-dimensional
locus for $G_2$-instantons. This is a subtle analytic issue in the
Donaldson--Thomas--Segal program \cite{DonaldsonSegal:11}.
\end{itemize} 

We now describe the main results and the organization of the paper.

The first main result, proved as Theorem~\ref{thm:bubbling-rigidity}, is a
local rigidity identity when a family degenerates to a rank-two torsion free sheaf with a point singularity. In the local
setting of Section~\ref{TwoNotionsOfM}, let $\Ecal$ be a one-parameter
degeneration over $B\times \Delta$ whose general fibers are rank-two locally
free sheaves, and the central fiber $\Ccal$ is torsion-free with an
isolated point singularity. Then
$$
l(\Ccal^{**}/\Ccal)=\frac{1}{2}l(\Ecal xt^1(\Ccal,\Ocal_B))>0.
$$
Motivated by singularity formation on surfaces in gauge theory, we call $m^{\mathrm{alg}}:=l(\Ccal^{**}/\Ccal)$ the algebraic bubbling
multiplicity inserted into the central fiber. Also
$m^{\mathrm{alg}}_{0}:=l(\Ecal xt^1(\Ccal,\Ocal_B))$ depends only on the reflexive hull $\Ccal^{**}$. Thus isolated point bubbling in rank two
is rigid: the bubbling multiplicity is forced by the singularity of the
reflexive hull. In particular, the central fiber is never reflexive, its
reflexive hull is not locally free, and the Ext-length is even. Furthermore, this suggests that one should not expect a uniform $L^2$ estimate for
all limiting sections when a sequence of HYM connections develops an
isolated point singularity. These
restrictions provide the numerical input for the smoothability problems
studied next. 

The next result shows that the numerical obstruction furnished by the
rigidity identity is sharp, but not purely numerical. Fix a rank-two
reflexive sheaf $\Fcal$ with an isolated point singularity. The identity
prescribes the length of any elementary modification $\Ccal\subset \Fcal$
that can occur as the central fiber of a smoothing; it does not prescribe
the quotient $\Fcal\twoheadrightarrow \Fcal/\Ccal$. In
Section~\ref{DeformingSingularities} we construct an explicit class of
singularities for which the prescribed length is attained: there exists a
quotient $\Fcal\twoheadrightarrow \tau$ of that length such that
$\Kcal er(\Fcal\to \tau)$ is the central fiber of a flat family with locally free
general fiber (see Proposition~\ref{thm-explicit-smoothable}). We then show
that the choice of quotient matters: a fixed reflexive sheaf may admit more
than one smoothable elementary modification. In the minimal case
$m^{\mathrm{alg}}=1$, however, the smoothable model is unique up to local
analytic coordinate change (see Theorem~\ref{thm:minimal-bubbling-model}). This unique minimal model is
the local algebraic model underlying the example discovered in \cite{ChenSun:18},
and it is the model globalized in the following section.

The minimal local model is realized globally in
Section~\ref{Section-Global appearance of the minimal singularity and bubbling}.
Theorem~\ref{thm:minimal-global-degeneration} constructs a flat family on
$\PBbb^3$ whose general fibers are stable rank two vector bundles and whose central
fiber is torsion free with two isolated minimal point singularities. Via the
Donaldson--Uhlenbeck--Yau theorem, this algebraic family gives a degeneration
of smooth Hermitian--Yang--Mills connections. Its Uhlenbeck limit is the
admissible connection on the reflexive hull of the central fiber and carries
no codimension-two blow-up cycle; nevertheless, the analytic tangent cone at
each point singularity has flat connection part and a multiplicity-one line as
its blow-up cycle. Thus the local example discovered in \cite{ChenSun:18}
appears as an actual degeneration of smooth HYM connections.

Rescaling at a singular point then produces one of the main analytic
objects of the paper: a centered minimal-density HYM bubble on
$\mathbb C^3$. This is a smooth non-flat HYM connection with density one
at infinity. Its tangent cone at infinity has flat connection part, but
the tangent-cone data still contains a nontrivial multiplicity-one complex
line as the blow-up cycle. The section ends with a
conditional compactification result: assuming the expected compatibility
between algebraic compactifications and tangent-cone data at infinity, the
boundary sheaf at infinity is forced into two explicit types, and the split
type is excluded for algebraic bubbles arising from families forming the
minimal singularity. Assuming the folklore expectation that the analytic bubble has split type at infinity, this gives a precise sense in which the centered analytic bubble differs from the algebraic bubble obtained by blowing up the degenerating family.

The final section shows that the preceding phenomena are not confined to
the minimal model.  Section~\ref{GlobalExamples} constructs a systematic
class of smoothable elementary modifications of projective-cone
singularities with explicitly computable algebraic bubbles. Starting from a
rank-two stable vector bundle $\underline{\Gcal}$ on $\PBbb^2$ with
$c_1(\underline{\Gcal})=0$, we form the associated reflexive cone $\Fcal$ on
$\PBbb^3$ and construct a flat family $\Ecal$ whose general fiber is locally
free and whose central fiber is an elementary modification
$\Ccal\subset \Fcal$ at the vertex. In this class, the algebraic bubble can be computed explicitly.

These projective-cone examples also provide explicit analytic local models.
The admissible HYM connection on the cone is induced by the HYM connection
on $\underline{\Gcal}$, and in the affine chart near the vertex it becomes
the standard HYM cone connection. Thus the examples give a large supply of
local models for isolated point bubbling in complex dimension three,
separate from the familiar bubbling along complex codimension-two loci. They
also show that algebraic bubbling multiplicity and analytic energy density
are not governed by a simple universal identity: as illustrated by
Example~\ref{NoRelation}, the two quantities record different features of
the same degeneration. Finally, after taking the product with $\RBbb$, these
models yield local models for bubbling along real one-dimensional loci for
$G_2$-instantons.

\subsection*{Notations and conventions} In this paper, we fix a few local notations. Let
$B\subset \CBbb^3$ be the three-dimensional unit ball centered at the
origin $o\in B$, let $\Delta\subset \CBbb$ be the unit disk centered at
the origin, and set
$$
        X=B\times \Delta,\qquad B_t=B\times\{t\}.
$$
We write $\Ocal=\Ocal_X$ for the structure sheaf of $X$ and
$\Ocal_t=\Ocal_{B_t}$ for the structure sheaf of $B_t$. By a family we
mean a coherent sheaf on $X$, usually denoted by $\Ecal$. Its fiber over
$t$ is
$$
        \Ecal_t:=\Ecal\otimes_{\Ocal}\Ocal_t .
$$
The subscript $t$ always denotes restriction to the fiber $B_t$. But to avoid confusion with the stalk of sheaf, we use $\Ccal$ to denote the central fiber over $t=0$, i.e.
$$
\Ccal:=\Ecal \otimes _\Ocal \Ocal_B
$$
Similarly, for other sheaves $\Hcal$ on $X$, we directly use $\Hcal|_B$ to denote the restriction to $B\times \{0\}$. We also use $\Fcal$ to denote a coherent sheaf on $B$ in general.

If $\tau$ is a finite-length sheaf supported at $o\in B$ or at
$(o,0)\in X$, we write
$$
        l(\tau):=\dim_{\CBbb}H^0(\tau).
$$

Also for later use, if $\Gcal$ is a sheaf with an isolated singularity, we write $hd(\Gcal)$ and $depth(\Gcal)$ for the homological
dimension and the depth of stalk of $\Gcal$ at the singularity.

Finally, $\iota:B_0\hookrightarrow X$ denotes the natural inclusion. When
the central fiber $\Ccal$ is regarded as a sheaf on $X$, we write
$\iota_*\Ccal$.

\subsection*{Acknowledgments}
The author thanks Richard Thomas for helpful comments on an earlier version of this manuscript and for pointing out relevant references. He also thanks Song Sun for helpful discussions related to this work, Yang Li and Keshu Zhou for helpful comments, and Aleksander Doan, Richard Wentworth, and  Junsheng Zhang for their interest in this work. This work was partially supported by NSERC and the ECR Supplement.

\section{Rigid bubbling: an identity between two notions of algebraic multiplicity}\label{TwoNotionsOfM}
\subsection{Main results}
\begin{defi}
A family $\Ecal$ is said to \emph{degenerate to a rank-two torsion-free
sheaf with an isolated point singularity at the origin} if $\Ecal_t$ is
locally free of rank two for every $t\neq 0$, and the central fiber $\Ccal$ is a
rank-two torsion-free sheaf which is locally free away from the origin
but not locally free at the origin.

A rank-two torsion-free sheaf $\Fcal$ on $B$ with an isolated singularity
at the origin is called \emph{smoothable} if
$$
        \Fcal\cong \Ccal
$$
where $\Ccal$ is the central fiber of some such family $\Ecal$.

A subsheaf $\Gcal\subset \Fcal$ is called an \emph{elementary
modification} of $\Fcal$ at the origin if $\Fcal/\Gcal$ is supported at
the origin.
\end{defi}

\begin{rmk}
If a family degenerates to a rank-two torsion-free sheaf with an
isolated point singularity at the origin, then the family must be
reflexive; see Lemma~\ref{ReflexiveProperty}. In particular, it is flat
over $\Delta$.
\end{rmk}
 
Fix, from now on, a family $\Ecal$ degenerating to a rank-two torsion-free sheaf with an isolated point singularity at the origin. We attach to its central fiber $\Ccal$ the following two
finite-length invariants:
$$
        m^{\mathrm{alg}}
        :=
        l\bigl(\Ccal^{**}/\Ccal\bigr),
        \qquad
        m_0^{\mathrm{alg}}
        :=
        l\bigl(\Ecal xt^1(\Ccal,\Ocal_B)\bigr).
$$

The first invariant is motivated by gauge theory
\cite{Chen2025, DonaldsonKronheimer:90, GSTW:18}; we call it the
\emph{algebraic bubbling multiplicity} of the family at the origin. If
$m^{\mathrm{alg}}\neq 0$, we say that the family exhibits algebraic
bubbling of multiplicity $m^{\mathrm{alg}}$ at the origin.

The second invariant is less familiar from the analytic point of view,
but is standard in the algebraic study of singularities of reflexive
sheaves \cite{Hartshorne:1980}. In the present situation,
$$
        \Ecal xt^1(\Ccal,\Ocal_B)
        \cong
        \Ecal xt^1(\Ccal^{**},\Ocal_B).
$$
Hence
$$
        m_0^{\mathrm{alg}}
        =
        l\bigl(\Ecal xt^1(\Ccal^{**},\Ocal_B)\bigr),
$$
so $m_0^{\mathrm{alg}}$ depends only on the reflexive hull
$\Ccal^{**}$ and measures its defect of local freeness at the origin.

\begin{thm}[Bubbling rigidity]\label{thm:bubbling-rigidity}
Suppose $\Ecal$ is a family degenerating to a rank-two torsion-free sheaf
with an isolated point singularity at the origin. Then
$$
m_0^{\mathrm{alg}}
=
2m^{\mathrm{alg}}.
$$
Furthermore,
$$
m_0^{\mathrm{alg}}\equiv 0 \pmod{2},
\qquad
m^{\mathrm{alg}}>0.
$$
In particular, $\Ccal$ is not reflexive, and its reflexive hull
$\Ccal^{**}$ is not locally free at the origin.
\end{thm}

\begin{rmk}\label{Rmk: local c3 view}
\begin{itemize}
\item The invariant $l\bigl(\Ecal xt^1(\Fcal,\Ocal_B)\bigr)$
is the standard local invariant measuring an isolated singularity of a
rank-two reflexive sheaf on a smooth threefold, already appearing in
Hartshorne's work on stable reflexive sheaves
\cite{Hartshorne:1980}. In a global projective setting, this invariant is
one of the local terms entering the third Chern class formalism for
torsion-free sheaves. If a global degeneration has locally free rank-two
general fibers and its central fiber has a single isolated singular point,
then the global vanishing of the third Chern class can be read locally as
a cancellation between the contribution of the reflexive hull and the
zero-dimensional quotient $\Ccal^{**}/\Ccal$. If the central fiber has
several isolated singular points, the global vanishing gives only a
summed relation over all points, not a pointwise balance. Theorem
\ref{thm:bubbling-rigidity} supplies precisely this pointwise local
identity, without relying on a global Chern class argument. The local
third Chern class formalism for rank-two torsion-free sheaves with
isolated point singularities is developed systematically in a separate
paper (\cite{Chen26b}).

\item Perrin obtained related exact sequences for reduced limits of
rank-two vector bundles in a global moduli-theoretic setting on smooth
projective threefolds \cite{Perrin:2001}. In the isolated-point case,
under Perrin's hypotheses, his exact sequence implies a factor-two
length relation. The reducedness condition, however, is a condition on
the total family. Locally, if $t$ is the degeneration parameter, it
requires the higher Ext sheaves
$$
\Ecal xt^i_{\Ocal_X}(\Ecal,\Ocal_X),\qquad i>0,
$$
to be annihilated by $t$. We do not impose this condition, and the
local families considered here are in general non-reduced in this
sense; see Remark~\ref{Rmk:Nonreducedness} for concrete examples.
Moreover, Perrin's framework is global and purely algebraic, whereas the
present paper isolates the factor-two relation in a local point-bubbling
setting and allows families not covered by the global reduced-limit
framework.
\end{itemize}
\end{rmk}
The rigidity identity has several consequences and applications.

\begin{itemize}
\item Let $\Fcal$ be a rank-two reflexive sheaf on $B$, locally free on
$B\setminus o$. If
$$
        l\bigl(\Ecal xt^1(\Fcal,\Ocal_B)\bigr)\equiv 1 \pmod{2},
$$
then $\Fcal$ cannot occur as the reflexive hull of the central fiber of
an isolated point degeneration. Equivalently, there is no family $\Ecal$
degenerating as above such that $\Ccal^{**}\cong \Fcal.$ For example, let $a,b,c\geq 1$ and let $\Fcal$ be defined by
$$
0\to \Ocal_B
\xrightarrow{
\begin{pmatrix}
z_1^a\\
z_2^b\\
z_3^c
\end{pmatrix}}
\Ocal_B^{\oplus 3}
\to \Fcal
\to 0 .
$$
If $abc$ is odd, there is no family $\Ecal$ degenerating as above
with $\Ccal^{**}\cong \Fcal$; see
Corollary~\ref{SingularityThatCannotBeDeformed}. This obstruction is
specific to isolated point degenerations: Example~\ref{HiddenSingularities}
shows that such singularities may still occur in degenerations with
codimension-two bubbling.

\item A family degenerating as above cannot, locally near $(o,0)$, admit a
two-term free presentation
$$
        0\to \Ocal^{\oplus r}
        \to \Ocal^{\oplus r+2}
        \to \Ecal
        \to 0
$$
for any $r$; see Corollary~\ref{NotFreeQuotient}. This contrasts with
rank-two reflexive sheaves on smooth threefolds, which always admit such
a local presentation. The difference is that here the family lives on the
fourfold $X=B\times\Delta$, where reflexivity alone does not force
homological dimension one.

\item The identity also has a concrete interpretation in terms of
sections. Since $\Ccal^{**}/\Ccal\neq 0,$
the reflexive hull $\Ccal^{**}$ contains local sections near the
origin which are not sections of $\Ccal$. On the punctured ball these
are sections of $\Ccal$, but they do not extend across the origin as
sections of the central fiber. Thus not every section of the limiting
reflexive hull can arise as an algebraic limit of sections of the nearby
fibers. This is reflected in the global examples (see Section \ref{Section-Global appearance of the minimal singularity and bubbling}), where stable
rank-two bundles develop isolated point singularities. It also suggests
why one should not expect a uniform $L^2$ estimate for
all limiting sections when a sequence of HYM connections develops an
isolated point singularity. It would be interesting to characterize
analytically which sections do arise as limits.

\item The rank-two assumption is essential. In higher rank, a direct
analogue of this rigidity identity need not hold. There are examples in
which bubbling occurs for $\Ecal^*$ but not for $\Ecal$ (see Example \ref{BubbleNotSymmetric}), revealing an
asymmetry between a family and its dual. 
\end{itemize}

Throughout this section, $\Ecal$ denotes a family degenerating to a
rank-two torsion-free sheaf $\Ccal$ with an isolated point singularity at the
origin. 
\subsection{Elementary properties of the two algebraic multiplicities}
\begin{lem}\label{ReflexiveProperty}
The sheaf $\Ecal$ is reflexive. In particular, $\Ecal$ is flat over
$\Delta$.
\end{lem}

\begin{proof}
We first show that $\Ecal$ is torsion-free. Let
$\kappa\subset \Ecal$ be the torsion subsheaf, and set
$\Qcal=\Ecal/\kappa$. Since $\Ecal$ is locally free away from $(o,0)$,
the sheaf $\kappa$ is supported at $(o,0)$ and hence has finite length.
Tensoring
$$
        0\to \kappa\to \Ecal\to \Qcal\to 0
$$
with $\Ocal_B$ gives
$$
        \Tcal or_1(\Ocal_B,\Qcal)
        \to
        \kappa/t\kappa
        \to
        \Ccal .
$$
Since $\Qcal$ is torsion-free, multiplication by $t$ is injective on
$\Qcal$, and hence
$$
        \Tcal or_1(\Ocal_B,\Qcal)=0 .
$$
Thus $\kappa/t\kappa$ injects into $\Ccal$. If $\kappa\neq 0$, then
$\kappa/t\kappa\neq 0$ by Nakayama's lemma. This gives a nonzero
finite-length subsheaf of the torsion-free sheaf $\Ccal$, a
contradiction. Therefore $\kappa=0$, and $\Ecal$ is torsion-free.

Now set
$$
        \tau=\Ecal^{**}/\Ecal .
$$
Since $\Ecal$ is locally free away from $(o,0)$, the sheaf $\tau$ is
supported at $(o,0)$ and has finite length. Tensoring
$$
        0\to \Ecal\to \Ecal^{**}\to \tau\to 0
$$
with $\Ocal_B$ gives an injection
$$
        \Tcal or_1(\Ocal_B,\tau)\hookrightarrow \Ccal,
$$
because $\Ecal^{**}$ is torsion-free and hence
$$
        \Tcal or_1(\Ocal_B,\Ecal^{**})=0 .
$$
The sheaf $\Tcal or_1(\Ocal_B,\tau)$ has finite length, so its injection
into the torsion-free sheaf $\Ccal$ forces
$$
        \Tcal or_1(\Ocal_B,\tau)=0 .
$$
Equivalently, multiplication by $t$ on $\tau$ is injective. On the other
hand, $\tau$ is supported at $(o,0)$, so multiplication by $t$ acts
nilpotently on $\tau$. Hence $\tau=0$. Therefore $\Ecal=\Ecal^{**}$ is
reflexive.

Finally, $\Ecal$ has no $t$-torsion. Since $\Delta$ is a smooth curve,
this implies that $\Ecal$ is flat over $\Delta$.
\end{proof}

Now we discuss some elementary properties of the algebraic
bubbling multiplicity of the family $\Ecal$. Recall that this is the
finite-length invariant of the central fiber $\Ccal$ defined by
$$
        m^{\mathrm{alg}}
        :=
        l\bigl(\Ccal^{**}/\Ccal\bigr).
$$

\begin{lem}\label{lem:first-multiplicity}
The algebraic bubbling multiplicity satisfies
$$
        m^{\mathrm{alg}}
        =
        l\bigl(\Kcal er(t:\Ecal xt^1(\Ecal, \Ocal) \to \Ecal xt^1(\Ecal, \Ocal))\bigr).
$$
\end{lem}
      
\begin{proof}
Choose finitely many local generators for $\Ecal^*$ near $(o,0)$. This
gives an exact sequence
$$
        0\rightarrow \Gcal \rightarrow \Ocal^n
        \rightarrow \Ecal^* \rightarrow 0 .
$$
Applying $\Hcal om_{\Ocal}(\bullet,\Ocal)$, we obtain
$$
        0\rightarrow \Ecal \rightarrow \Ocal^n
        \rightarrow \Gcal^*
        \rightarrow \Ecal xt^1(\Ecal^*,\Ocal)
        \rightarrow 0 .
$$
Set
$$
        \tau=\Ecal xt^1(\Ecal^*,\Ocal).
$$
After shrinking near $(o,0)$ if necessary, fix a trivialization of
$\det\Ecal$. Since $\Ecal$ has rank two, this identifies
$\Ecal^*$ with $\Ecal$. Hence
$$
        \tau\cong \Ecal xt^1(\Ecal,\Ocal).
$$

Let $\Hcal$ be the image sheaf of the map $\Ocal^n\to \Gcal^*$. Since
$\Hcal\subset \Gcal^*$, the sheaf $\Hcal$ has no $t$-torsion. Therefore
restricting
$$
        0\to \Ecal\to \Ocal^n\to \Hcal\to 0
$$
to the central fiber gives an exact sequence
$$
        0\to \Ccal \rightarrow \Ocal_B^n
        \rightarrow \left.\Hcal\right|_B \rightarrow 0 .
$$
Let $\operatorname{tor}(\left.\Hcal\right|_B)$ be the torsion of $\left.\Hcal\right|_B$. The inverse image of
$\operatorname{tor}(\left.\Hcal\right|_B)$ in $\Ocal_B^n$ is the
saturation of $\Ccal$ in $\Ocal_B^n$. Since this saturation is
reflexive and agrees with $\Ccal$ away from the origin, it is
$\Ccal^{**}$. Thus
$$
        \operatorname{tor}(\left.\Hcal\right|_B)
        \cong
        \Ccal^{**}/\Ccal,
$$
and hence
$$
        m^{\mathrm{alg}}
        =
        l\bigl(\operatorname{tor}(\left.\Hcal\right|_B)\bigr).
$$

Restricting
$$
        0\to \Hcal\to \Gcal^*\to \tau\to 0
$$
to $B$ gives
$$
        0\rightarrow \Tcal or_1(\tau,\Ocal_B)
        \rightarrow \left.\Hcal\right|_B
        \rightarrow \left.\Gcal^*\right|_B
        \rightarrow \left.\tau\right|_B
        \rightarrow 0 .
$$
Since $\Gcal^*$ is reflexive on the smooth fourfold $X$, its restriction
$\left.\Gcal^*\right|_B$ is torsion-free (\cite[Lemma $3.23$]{ChenSun:20a}). Therefore
$$
        \operatorname{tor}(\left.\Hcal\right|_B)
        \cong
        \Tcal or_1(\tau,\Ocal_B).
$$

Finally, by definition
$$
        \Tcal or_1(\tau,\Ocal_B)
        =
        \Kcal er(t:\tau\rightarrow\tau).
$$
Combining this with
$\tau\cong \Ecal xt^1(\Ecal,\Ocal)$, we obtain
$$
        m^{\mathrm{alg}}
        =
        l\bigl(
        \Kcal er(t:\Ecal xt^1(\Ecal,\Ocal)
        \to \Ecal xt^1(\Ecal,\Ocal))
        \bigr).
$$
This proves the lemma.
\end{proof}

    The positivity of $m^{\mathrm{alg}}$ follows from the following
nonvanishing.

\begin{prop}\label{ExtensionNonzero}
$\Ecal xt^1(\Ecal,\Ocal)\neq 0$. In particular, $m^{\mathrm{alg}}>0$.
\end{prop}

\begin{proof}
Let $\tau=\Ecal xt^1(\Ecal,\Ocal)$. Since $\Ecal$ is locally free away
from $(o,0)$, the sheaf $\tau$ is supported at $(o,0)$. Suppose for
contradiction that $\tau=0$.

After shrinking near $(o,0)$ if necessary, fix a trivialization of
$\det\Ecal$. Since $\Ecal$ has rank two, this identifies $\Ecal^*$ with
$\Ecal$. Hence
$$
        \Ecal xt^1(\Ecal^*,\Ocal)=0.
$$
Moreover, since $\Ecal^*$ is reflexive on the smooth fourfold $X$, we have
$hd(\Ecal^*)\leq 2$. Choose a locally free resolution
$$
        0\to \Ocal^{a}\to \Ocal^{b}\to \Ocal^{c}\to \Ecal^*\to 0.
$$
Applying $\Hcal om_{\Ocal}(-,\Ocal)$ and using
$\Ecal xt^1(\Ecal^*,\Ocal)=0$ gives an exact sequence
$$
        0\to \Ecal\to \Ocal^{c}\to \Ocal^{b}\to \Ocal^{a}\to \Hcal\to 0,
$$
where $\Hcal=\Ecal xt^2(\Ecal^*,\Ocal)$ is supported at $(o,0)$. From this, we know 
$$
hd(\Ecal)\leq 1.
$$
Thus, near $(o,0)$, there is a locally free presentation
$$
        0\to \Ocal^{r}\xrightarrow{\phi}\Ocal^{s}\to \Ecal\to 0.
$$
Applying $\Hcal om_{\Ocal}(-,\Ocal)$ gives
$$
        \Ecal xt^1(\Ecal,\Ocal)\cong
        (\Ocal^{r})^*/\Img(\phi^*).
$$
Since $\Ecal xt^1(\Ecal,\Ocal)=0$, the map $\phi^*$ is surjective. Hence
$\phi$ is a split injection, and $\Ecal$ is locally free near $(o,0)$.
This contradicts the assumption that the central fiber $\Ccal$ is not
locally free at the origin. Therefore $\Ecal xt^1(\Ecal,\Ocal)\neq 0$.

Finally, by Lemma~\ref{lem:first-multiplicity},
$$
        m^{\mathrm{alg}}
        =
        l\bigl(\Kcal er(t:\tau\to\tau)\bigr).
$$
If $m^{\mathrm{alg}}=0$, then multiplication by $t$ on $\tau$ is
injective. Since $\tau$ is supported at $(o,0)$, multiplication by $t$ is
nilpotent. Hence $\tau=0$, contradicting the first part. Therefore
$m^{\mathrm{alg}}>0$.
\end{proof}
        
     The positivity of $m^{\mathrm{alg}}$ also puts a restriction on the total
sheaf $\Ecal$.

\begin{cor}\label{NotFreeQuotient}
The sheaf $\Ecal$ admits no presentation of the form
$$
        0\to \Ocal^r\to \Ocal^{r+2}\to \Ecal\to 0 
$$
near the singular point.
\end{cor}

\begin{proof}
Assume that such a presentation exists. Since $\Ecal$ is flat over
$\Delta$, tensoring with $\Ocal_B$ gives an exact sequence
$$
        0\to \Ocal_B^r
        \xrightarrow{\Phi_0}
        \Ocal_B^{r+2}
        \to \Ccal
        \to 0 .
$$
Thus $hd(\Ccal)\leq 1$ and
$$
        depth(\Ccal)=3-hd(\Ccal)\geq 2 .
$$
The sheaf $\Ccal$ is torsion-free and locally free away from the origin, hence
the depth estimate shows that $\Ccal$ satisfies Serre's $S_2$ condition. Therefore $\Ccal$
is reflexive. This implies $m^{\mathrm{alg}}=l(\Ccal^{**}/\Ccal)=0$,
contradicting Proposition~\ref{ExtensionNonzero}.
\end{proof}
 
Recall the second algebraic invariant using only
the central fiber $\Ccal$:
$$
m^{\mathrm{alg}}_0:=l(\Ecal xt^1(\Ccal,\Ocal_B)).
$$

The next elementary property is well known and its proof is included for
completeness.

\begin{lem}\label{Lemma-Defect only depends on reflexive hull}
The algebraic multiplicity $m_0^{\mathrm{alg}}$ depends only on the reflexive hull
$\Ccal^{**}$. More precisely,
$$
m_0^{\mathrm{alg}}=l(\Ecal xt^1(\Ccal^{**},\Ocal_B)).
$$
\end{lem}

\begin{proof}
Set $Q:=\Ccal^{**}/\Ccal$. Since $B$ is smooth of dimension three and $Q$ is
supported at the origin, $Q$ is a finite-length sheaf of codimension three.
Hence
$$
\Ecal xt^k(Q,\Ocal_B)=0
$$
for all $k<3$. Apply $\mathcal{H}om_{\Ocal_B}(-,\Ocal_B)$ to the exact sequence
$$
0\rightarrow \Ccal
\rightarrow \Ccal^{**}
\rightarrow Q
\rightarrow 0.
$$
The vanishing above gives a natural isomorphism
$$
\Ecal xt^1(\Ccal^{**},\Ocal_B)
\cong
\Ecal xt^1(\Ccal,\Ocal_B).
$$
Taking lengths gives the desired identity.
\end{proof}
\subsection{Proof of Theorem \ref{thm:bubbling-rigidity}}

The proof begins with two elementary restriction facts.

\begin{lem}\label{RestrictionOfResolution}
A locally free resolution of $\Ecal$ restricts to a locally free resolution of
the central fiber $\Ccal$.
\end{lem}

\begin{proof}
Take a locally free resolution
$$
0\rightarrow \Ocal^a \xrightarrow{\Phi^1} \Ocal^b
\xrightarrow{\Phi^2} \Ocal^c \rightarrow \Ecal\rightarrow 0.
$$
Set $\Gcal:=\operatorname{Im}(\Phi^2)$. This gives two short exact sequences
$$
0\rightarrow \Ocal^a \rightarrow \Ocal^b \rightarrow \Gcal \rightarrow 0
$$
and
$$
0\rightarrow \Gcal \rightarrow \Ocal^c \rightarrow \Ecal \rightarrow 0.
$$
The sheaf $\Gcal$ is torsion free because it is a subsheaf of $\Ocal^c$.
Since $B$ is cut out by $t=0$, the sheaf
$\Tcal or_1(\iota_*\Ocal_B,\Fcal)$ is the $t$-torsion of $\Fcal$. Hence
$$
\Tcal or_1(\iota_*\Ocal_B,\Gcal)=0,
\qquad
\Tcal or_1(\iota_*\Ocal_B,\Ecal)=0.
$$
Therefore both short exact sequences remain exact after restricting to $B$.
Splicing the restricted sequences gives
$$
0\rightarrow \Ocal_B^a \xrightarrow{\Phi^1_0} \Ocal_B^b
\xrightarrow{\Phi^2_0} \Ocal_B^c \rightarrow \Ccal\rightarrow 0,
$$
which is a locally free resolution of $\Ccal$.
\end{proof}

The second restriction fact identifies the corresponding sheaf Ext groups.

\begin{lem}\label{SameExt}
For every $k$, $\Ecal xt^k(\Ecal,\iota_*\Ocal_B)
\cong
\iota_*\Ecal xt^k(\Ccal,\Ocal_B).$
\end{lem}

\begin{proof}
By Lemma \ref{RestrictionOfResolution}, choose a locally free resolution
$$
0\rightarrow \Ocal^a \xrightarrow{\Phi^1} \Ocal^b
\xrightarrow{\Phi^2} \Ocal^c \rightarrow \Ecal\rightarrow 0
$$
whose restriction to $B$ is
$$
0\rightarrow \Ocal_B^a \xrightarrow{\Phi^1_0} \Ocal_B^b
\xrightarrow{\Phi^2_0} \Ocal_B^c \rightarrow \Ccal\rightarrow 0.
$$
The complex computing $\Ecal xt^k(\Ecal,\iota_*\Ocal_B)$ is
$$
\mathcal{H}om_{\Ocal}(\Ocal^c,\iota_*\Ocal_B)
\xrightarrow{(\Phi^{2})^t}
\mathcal{H}om_{\Ocal}(\Ocal^b,\iota_*\Ocal_B)
\xrightarrow{(\Phi^{1})^t}
\mathcal{H}om_{\Ocal}(\Ocal^a,\iota_*\Ocal_B).
$$
For each $r$, there is a natural identification
$$
\mathcal{H}om_{\Ocal}(\Ocal^r,\iota_*\Ocal_B)
\cong
\iota_*\mathcal{H}om_{\Ocal_B}(\Ocal_B^r,\Ocal_B).
$$
Under these identifications, the complex above is the pushforward of
$$
\mathcal{H}om_{\Ocal_B}(\Ocal_B^c,\Ocal_B)
\xrightarrow{(\Phi^{2}_0)^t}
\mathcal{H}om_{\Ocal_B}(\Ocal_B^b,\Ocal_B)
\xrightarrow{(\Phi^{1}_0)^t}
\mathcal{H}om_{\Ocal_B}(\Ocal_B^a,\Ocal_B).
$$
Since $\iota_*$ is exact for a closed immersion, taking cohomology gives
$$
\Ecal xt^k(\Ecal,\iota_*\Ocal_B)
\cong
\iota_*\Ecal xt^k(\Ccal,\Ocal_B).
$$
\end{proof}

The proof also uses the following length computation.

\begin{lem}\label{LengthExtThree}
Let $\tau$ be a finite-length sheaf on $B$ supported at the origin. Then
$$
l(\Ecal xt^3(\tau,\Ocal_B))=l(\tau).
$$
\end{lem}

\begin{proof}
The proof is by induction on $l(\tau)$.

If $l(\tau)=1$, then $\tau\cong \Ocal_B/\mathfrak m$, where
$\mathfrak m$ is the maximal ideal of the origin. The Koszul complex of
$\mathfrak m$ gives a locally free resolution of $\Ocal_B/\mathfrak m$,
and its self-duality gives
$$
\Ecal xt^3(\Ocal_B/\mathfrak m,\Ocal_B)
\cong
\Ocal_B/\mathfrak m.
$$
Thus the claim holds when $l(\tau)=1$.

Assume $l(\tau)>1$. Choose a nonzero proper subsheaf
$\tau'\subset \tau$, and set $\tau'':=\tau/\tau'$. Then
$$
0\rightarrow \tau'\rightarrow \tau\rightarrow \tau''\rightarrow 0
$$
is exact, with $0<l(\tau'),l(\tau'')<l(\tau)$. Since $B$ is smooth of
dimension three and all three sheaves have finite length, the associated
long exact sequence gives
$$
0\rightarrow \Ecal xt^3(\tau'',\Ocal_B)
\rightarrow \Ecal xt^3(\tau,\Ocal_B)
\rightarrow \Ecal xt^3(\tau',\Ocal_B)
\rightarrow 0.
$$
Taking lengths and applying the induction hypothesis gives
$$
l(\Ecal xt^3(\tau,\Ocal_B))
=
l(\Ecal xt^3(\tau',\Ocal_B))
+
l(\Ecal xt^3(\tau'',\Ocal_B))
=
l(\tau')+l(\tau'')
=
l(\tau).
$$
\end{proof}
   
    \begin{proof}[Proof of Theorem \ref{thm:bubbling-rigidity}]
Apply $\mathcal{H}om(\Ecal,\bullet)$ to the short exact sequence
$$
0\rightarrow \Ocal \xrightarrow{t} \Ocal
\rightarrow \iota_*\Ocal_B \rightarrow 0.
$$
Put $\tau_i:=\Ecal xt^i(\Ecal,\Ocal)$ for $i=1,2$. The associated long
exact sequence contains
$$
0\rightarrow \tau_1/t\tau_1
\rightarrow \Ecal xt^1(\Ecal,\iota_*\Ocal_B)
\rightarrow \Kcal er(t:\tau_2\rightarrow \tau_2)
\rightarrow 0.
$$
The length-two locally free resolution of $\Ecal$ gives
$\Ecal xt^3(\Ecal,\Ocal)=0$, so the same long exact sequence gives
$$
\Ecal xt^2(\Ecal,\iota_*\Ocal_B)\cong \tau_2/t\tau_2.
$$

Under the isolated-singularity assumption, $\tau_1$ and $\tau_2$ have
finite length. For an endomorphism of a finite-length sheaf, the kernel and
cokernel have the same length. Using the earlier identity from Lemma \ref{lem:first-multiplicity}
$$
m^{\mathrm{alg}}
=
l\bigl(\Kcal er(t:\tau_1\rightarrow \tau_1)\bigr),
$$
the two displays above give
$$
l(\Ecal xt^1(\Ecal,\iota_*\Ocal_B))
=
m^{\mathrm{alg}}
+
l(\Ecal xt^2(\Ecal,\iota_*\Ocal_B)).
$$

By Lemma \ref{SameExt}, applied with $k=1,2$,
$$
l(\Ecal xt^1(\Ecal,\iota_*\Ocal_B))
=
l(\Ecal xt^1(\Ccal,\Ocal_B))
$$
and
$$
l(\Ecal xt^2(\Ecal,\iota_*\Ocal_B))
=
l(\Ecal xt^2(\Ccal,\Ocal_B)).
$$
Hence
$$
m_0^{\mathrm{alg}}
=
m^{\mathrm{alg}}
+
l(\Ecal xt^2(\Ccal,\Ocal_B)).
$$

It remains to compute the last length. Set $Q:=\Ccal^{**}/\Ccal$. Then
$Q$ is the finite-length defect and $m^{\mathrm{alg}}=l(Q)$. Applying
$\mathcal{H}om_{\Ocal_B}(-,\Ocal_B)$ to
$$
0\rightarrow \Ccal \rightarrow \Ccal^{**}\rightarrow Q\rightarrow 0
$$
gives the relevant part of the long exact sequence. Since $Q$ has finite
length on the smooth threefold $B$,
$\Ecal xt^i(Q,\Ocal_B)=0$ for $i<3$. Since $\Ccal^{**}$ is reflexive on
$B$, it has depth at least two; hence $hd(\Ccal^{**})\leq 1$. Thus
$\Ecal xt^i(\Ccal^{**},\Ocal_B)=0$ for $i\geq 2$. Therefore
$$
\Ecal xt^2(\Ccal,\Ocal_B)
\cong
\Ecal xt^3(Q,\Ocal_B).
$$
By Lemma \ref{LengthExtThree},
$$
l(\Ecal xt^2(\Ccal,\Ocal_B))
=
l(\Ecal xt^3(Q,\Ocal_B))
=
l(Q)
=
m^{\mathrm{alg}}.
$$
Substituting this into the previous identity gives
$$
m_0^{\mathrm{alg}}=2m^{\mathrm{alg}}.
$$
In particular, $l(\Ecal xt^1(\Ccal,\Ocal_B))$ is even.

By Proposition \ref{ExtensionNonzero}, $m^{\mathrm{alg}}>0$, hence $m_0^{\mathrm{alg}}>0$. By Lemma
\ref{Lemma-Defect only depends on reflexive hull},
$$
m_0^{\mathrm{alg}}
=
l(\Ecal xt^1(\Ccal^{**},\Ocal_B)).
$$
Thus $\Ecal xt^1(\Ccal^{**},\Ocal_B)\neq 0$, so $\Ccal^{**}$ is not locally
free at the origin.
\end{proof}
   
We record the following immediate consequence.

\begin{cor}\label{IdealNotSmoothable}
Let $\Ical\subset \Ocal_B$ be a proper finite-colength ideal sheaf
supported at the origin. Then the rank-two torsion-free sheaf
$\Ocal_B\oplus \Ical$ is not smoothable.
\end{cor}

\begin{proof}
Indeed, $(\Ocal_B\oplus \Ical)^{**}\cong \Ocal_B^{\oplus 2}$ is locally
free, contradicting Theorem~\ref{thm:bubbling-rigidity}.
\end{proof}

\begin{rmk}\label{rmk-2.14}
It is well known that the corresponding statement is false in complex dimension two. On a
neighbourhood of the origin in $\CBbb^2\times \Delta$, with coordinates
$x,y,z$ and parameter $z$, consider
$$
0\rightarrow \Ocal
\xrightarrow{
\begin{pmatrix}
x\\
y\\
z
\end{pmatrix}}
\Ocal^{\oplus 3}
\rightarrow \Ecal
\rightarrow 0 .
$$
For $z\neq 0$, the restriction of $\Ecal$ to the fiber $\{z=\mathrm{const}\}$
is locally free. On the central surface $S=\{z=0\}$, writing
$\Ccal=\left.\Ecal\right|_S$, one has
$$
        \Ccal\cong (x,y)\oplus \Ocal_S,
        \qquad
        \Ccal^{**}\cong \Ocal_S^{\oplus 2}.
$$
Thus $\Ocal_S\oplus (x,y)$ is smoothable on a surface, and taking the
double dual removes the isolated singularity.
\end{rmk}
We also record the following parity obstruction.

\begin{cor}\label{OddParityObstruction}
Let $\Fcal$ be a rank-two torsion-free sheaf on $B$ with an isolated
point singularity at the origin. If
$l(\Ecal xt^1(\Fcal,\Ocal_B))\equiv 1 \pmod{2}$, then $\Fcal$ is not
smoothable. In particular, if moreover $\Fcal$ is reflexive, then there is no family
$\Ecal$ degenerating as above such that $\Ccal^{**}\cong \Fcal$.
\end{cor}

\begin{proof}
If $\Fcal\cong \Ccal$ for such a family, then
$l(\Ecal xt^1(\Fcal,\Ocal_B))=m^{\mathrm{alg}}_0$, which is even by
Theorem~\ref{thm:bubbling-rigidity}. This proves the first assertion.

If $\Fcal$ is moreover reflexive with an essential isolated singularity and $\Ccal^{**}\cong \Fcal$ for some degenerating family $\Ecal$ as above, then
Lemma~\ref{Lemma-Defect only depends on reflexive hull} gives
$$
        l(\Ecal xt^1(\Fcal,\Ocal_B))
        =
        l(\Ecal xt^1(\Ccal,\Ocal_B))
        =
        m^{\mathrm{alg}}_0=2m^{\mathrm{alg}},
$$
which is a contradiction.
\end{proof}
   
The parity obstruction gives the following explicit class of reflexive
singularities which cannot occur as reflexive hulls of isolated point
degenerations.

\begin{cor}\label{SingularityThatCannotBeDeformed}
Let $a,b,c\geq 1$ with $abc\equiv 1 \pmod{2}$, and let $\Fcal$ be defined by
$$
0\rightarrow \Ocal_B
\xrightarrow{
\begin{pmatrix}
z_1^a\\
z_2^b\\
z_3^c
\end{pmatrix}}
\Ocal_B^{\oplus 3}
\rightarrow \Fcal
\rightarrow 0 .
$$
Then there is no family $\Ecal$ degenerating as above such that
$\Ccal^{**}\cong \Fcal$.
\end{cor}

\begin{proof}
The sheaf $\Fcal$ is reflexive: it is locally free away from the origin,
and at the origin it has homological dimension one, hence depth two. Applying $\Hcal om_{\Ocal_B}(-,\Ocal_B)$ to the defining sequence gives
$$
        \Ecal xt^1(\Fcal,\Ocal_B)
        \cong
        \Ocal_B/(z_1^a,z_2^b,z_3^c).
$$
Thus $l(\Ecal xt^1(\Fcal,\Ocal_B))=abc\equiv 1 \pmod{2}$. The claim
follows from Corollary~\ref{OddParityObstruction}.
\end{proof}

Although the reflexive singularities above cannot occur as reflexive
hulls of isolated point degenerations, they can occur once a
codimension-two defect is allowed.  The following example makes this
explicit.

\begin{exam}\label{HiddenSingularities}
Let $\Ecal$ be defined by the short exact sequence
$$
0\rightarrow \Ocal^{\oplus 2}
\xrightarrow{
\begin{pmatrix}
z_1&t\\
z_2&0\\
-1&z_2\\
0&z_3
\end{pmatrix}}
\Ocal^{\oplus 4}
\rightarrow \Ecal
\rightarrow 0 .
$$
For $t\neq 0$, the defining matrix has rank two everywhere on $B_t$.
On $X$, its rank drops precisely along
$$
        z_2=z_3=t=0 .
$$
Thus this is not an isolated point degeneration; the non-locally-free
locus is a curve contained in the central fiber.

Since $\Ecal$ has a length-one locally free resolution and is locally
free in codimension one, Serre's $S_2$ criterion shows that $\Ecal$ is
reflexive.  Restricting to the central fiber gives
$$
0\rightarrow \Ocal_B^{\oplus 2}
\xrightarrow{
\begin{pmatrix}
z_1&0\\
z_2&0\\
-1&z_2\\
0&z_3
\end{pmatrix}}
\Ocal_B^{\oplus 4}
\rightarrow \Ccal
\rightarrow 0 .
$$
Equivalently, $\Ccal$ admits the presentation
$$
0\rightarrow \Ocal_B
\xrightarrow{
\begin{pmatrix}
z_1\\
z_2\\
z_3
\end{pmatrix}}
\Ocal_B^{\oplus 2}\oplus (z_2,z_3)
\rightarrow \Ccal
\rightarrow 0 .
$$
In particular, $\Ccal$ is torsion-free, and its reflexive hull is given by
$$
0\rightarrow \Ocal_B
\xrightarrow{
\begin{pmatrix}
z_1\\
z_2\\
z_3
\end{pmatrix}}
\Ocal_B^{\oplus 3}
\rightarrow \Ccal^{**}
\rightarrow 0 .
$$
Moreover,
$$
        \Ccal^{**}/\Ccal\cong \Ocal_B/(z_2,z_3),
$$
so the quotient is supported on the curve $\{z_2=z_3=0\}$.

Thus $\Ccal^{**}$ is the reflexive singularity from
Corollary~\ref{SingularityThatCannotBeDeformed} with $a=b=c=1$.  It
cannot occur as the reflexive hull of an isolated point degeneration, but
it appears here because the degeneration has a codimension-two
non-locally-free locus.
\end{exam}
       In higher rank, the rank-two rigidity fails.  An isolated point
singularity may form without algebraic bubbling, and algebraic bubbling
need not be symmetric under dualization: one of $\Ecal$ and $\Ecal^*$
may bubble while the other does not.  This contrasts with
codimension-two HYM bubbling, which is unchanged by passing to the dual
connection. The following example illustrates this.

   \begin{exam}\label{BubbleNotSymmetric}
Let
$$
        \mathfrak m_X=(z_1,z_2,z_3,t)\subset \Ocal,
        \qquad
        \mathfrak m_B=(z_1,z_2,z_3)\subset \Ocal_B .
$$
Consider the rank-three reflexive sheaf $\Ecal$ on $X$ defined by
$$
0\to \Ocal
\xrightarrow{(z_1,z_2,z_3,t)^T}
\Ocal^{\oplus 4}
\to \Ecal
\to 0 .
$$
Dualizing gives
$$
0\to \Ecal^*
\to \Ocal^{\oplus 4}
\xrightarrow{(z_1,z_2,z_3,t)}
\mathfrak m_X
\to 0 .
$$

The central fiber $\Ccal$ is reflexive.  Indeed,
$$
        \Ccal
        \cong
       \Ccal oker
        \bigl(\Ocal_B\xrightarrow{(z_1,z_2,z_3)^T}
        \Ocal_B^{\oplus 3}\bigr)
        \oplus \Ocal_B ,
$$
and the first summand is reflexive on the smooth threefold $B$.  Hence
$$
        l(\Ccal^{**}/\Ccal)=0,
$$
so there is no algebraic bubbling for $\Ecal$.

On the other hand, $\left.\Ecal^*\right|_B$ is not reflexive.  Since
$t$ is a nonzerodivisor on $\mathfrak m_X$, restricting the dual sequence
to $B$ gives
$$
0\to \left.\Ecal^*\right|_B
\to \Ocal_B^{\oplus 4}
\to
\mathfrak m_X/t\mathfrak m_X
\to 0 .
$$
As an $\Ocal_B$-module,
$$
        \mathfrak m_X/t\mathfrak m_X
        \cong
        \mathfrak m_B\oplus \CBbb\bar t,
$$
where $\bar t$ is the quotient class of $t$ and
$\CBbb\bar t$ is a copy of the skyscraper sheaf at the origin.  Thus
$$
        \left.\Ecal^*\right|_B
        \cong
        \Kcal er\bigl(\Ocal_B^{\oplus 3}\to \mathfrak m_B\bigr)
        \oplus \mathfrak m_B .
$$
The first summand is reflexive, while $\mathfrak m_B^{**}\cong \Ocal_B$.
Therefore
$$
        \bigl(\left.\Ecal^*\right|_B\bigr)^{**}
/       \left.\Ecal^*\right|_B
        \cong
        \Ocal_B/\mathfrak m_B,
$$
which has length one.  Hence algebraic bubbling occurs for the dual
family $\Ecal^*$.
\end{exam}

\section{Smoothable elementary modifications}\label{DeformingSingularities}
In this section, we construct smoothable central fibers by modifying
prescribed reflexive singularities at the origin.  The rigidity identity fixes the length of the quotient in any smoothable
elementary modification,
but not the modification itself.  We show that, for an explicit class of
reflexive singularities, this numerical prediction is sharp.  We then
examine the choice of modification: it is not unique in general, while in
the minimal case $m^{\mathrm{alg}}=1$ the smoothable central-fiber model
is unique up to local analytic coordinate change.

\subsection{Main results}
We first construct an explicit class of
elementary modifications which attain the bubbling multiplicity required by
Theorem~\ref{thm:bubbling-rigidity}.  Consider a sheaf $\Fcal$ defined by
$$
0\rightarrow \Ocal_B
\xrightarrow{
\begin{pmatrix}
f_1\\
f_2\\
f_3^2
\end{pmatrix}}
\Ocal_B^{\oplus 3}
\rightarrow \Fcal
\rightarrow 0,
$$
where
$$
        \sqrt{(f_1,f_2,f_3^2)}=(z_1,z_2,z_3).
$$
Let
$
        \tau=\Ocal_B/(f_1,f_2,f_3),
$
and consider the quotient map $q: \Fcal\to\tau$ induced by projection of
$\Ocal_B^{\oplus 3}$ to its third factor. 
\begin{prop}\label{thm-explicit-smoothable}
With the notation above, there exists a family $\Ecal$ degenerating to a central fiber $\Ccal$ with an
isolated point singularity at the origin such that
$$
\Ccal\cong\Kcal er(q).
$$
In particular,
$$
        \Ccal^{**}\cong \Fcal,
$$
where $\Fcal$ is the reflexive sheaf defined above.
\end{prop}

\begin{rmk}
   A fixed reflexive sheaf can
admit more than one smoothable elementary modification; see
Corollary~\ref{Cor-nonuniqueness}.
\end{rmk}

In the minimal case where $m^{\mathrm{alg}}=1$, the smoothable central-fiber
model is unique up to local analytic coordinate change. To state the model, let $\Fcal_{\min}$ be the rank-two reflexive sheaf defined by
$$
0\rightarrow \Ocal_B
\xrightarrow{
\begin{pmatrix}
z_1\\
z_2\\
z_3^2
\end{pmatrix}}
\Ocal_B^{\oplus 3}
\rightarrow \Fcal_{\min}
\rightarrow 0 .
$$
Let
$$
        q_{\min}:\Fcal_{\min}\rightarrow \Ocal_B/(z_1,z_2,z_3)
$$
be the quotient induced by projection of $\Ocal_B^{\oplus 3}$ onto its
third factor, and set
$$
        \Ccal_{\min}:=\Kcal er(q_{\min}).
$$

\begin{thm}[The minimal bubbling model]\label{thm:minimal-bubbling-model}
The elementary modification
$\Ccal_{\min}\subset \Fcal_{\min}$ is smoothable. Conversely, if $\Ecal$
is a family degenerating to an isolated point singularity at the origin
with $m^{\mathrm{alg}}=1$, then, after a local analytic change of
coordinates at the origin, the pair
$\Ccal\subset \Ccal^{**}$
is isomorphic to
$\Ccal_{\min}\subset \Fcal_{\min}.$
\end{thm}

\begin{rmk}
\begin{itemize}
    \item The reflexive sheaf $\Fcal_{\min}$ is the local model underlying the
example discovered in \cite{ChenSun:18}: the original admissible HYM
connection has a genuine essential singularity at the point, represented
on the sheaf side by the non-locally-free reflexive hull above, while its
analytic tangent-cone connection is flat. Thus the essential singularity
is not visible at the tangent-cone level. The global examples in the next section realize the same elementary
modification
$$
\Ccal_{\min}\subset \Fcal_{\min}
$$
as a degeneration of smooth HYM connections modulo gauge transformations.
\item The theorem also illustrates the subtlety of smoothability: the bubbling identity fixes the algebraic bubbling multiplicity, equivalently
the length of $\Ccal^{**}/\Ccal$, but not every elementary modification with
that length is smoothable.
\end{itemize}

\end{rmk}

\subsection{Proof of Proposition \ref{thm-explicit-smoothable}: a general smoothing construction}
Recall from the above that we consider three holomorphic functions $f_1,f_2,f_3$ defined near $0\in B$,
and assume that
$$
        \sqrt{(f_1,f_2,f_3^2)}=(z_1,z_2,z_3).
$$
Define a reflexive sheaf $\Fcal$ near $0\in B$ by
$$
0\rightarrow \Ocal_B
\xrightarrow{
\begin{pmatrix}
f_1\\
f_2\\
f_3^2
\end{pmatrix}}
\Ocal_B^{\oplus 3}
\rightarrow \Fcal
\rightarrow 0 .
$$
Then
$$
        \Ecal xt^1(\Fcal,\Ocal_B)
        \cong
        \Ocal_B/(f_1,f_2,f_3^2).
$$
The choice of $f_3^2$ is made so that the Ext-length is twice the length
of the elementary modification predicted by
Theorem~\ref{thm:bubbling-rigidity}.  Indeed, if $\Fcal$ occurs as the
reflexive hull $\Ccal^{**}$ of a smoothable central fiber $\Ccal$ in family $\Ecal$, then
Theorem~\ref{thm:bubbling-rigidity} forces
$$
        l(\Ccal^{**}/\Ccal)
        =
        \frac{1}{2}l\bigl(\Ecal xt^1(\Ccal,\Ocal_B)\bigr).
$$
In the present situation this required length is
$$
        l\bigl(\Ocal_B/(f_1,f_2,f_3)\bigr).
$$

This suggests the following elementary modification.  Define $\Ccal$ by
$$
0\rightarrow \Ocal_B
\xrightarrow{
\begin{pmatrix}
f_1\\
f_2\\
f_3^2
\end{pmatrix}}
\Ocal_B^{\oplus 2}\oplus (f_1,f_2,f_3)
\rightarrow \Ccal
\rightarrow 0 .
$$
Then
$$
        \Ccal^{**}/\Ccal
        \cong
        \Fcal/\Ccal
        \cong
        \Ocal_B/(f_1,f_2,f_3).
$$
Thus $\Ccal$ has exactly the algebraic bubbling multiplicity required by the
rigidity identity, and is the natural candidate to smooth.

To construct such a smoothing, we replace the ideal $(f_1,f_2,f_3)$ by
its Koszul complex
$$
0\rightarrow \Ocal_B
\xrightarrow{
\begin{pmatrix}
f_1\\
f_2\\
f_3
\end{pmatrix}}
\Ocal_B^{\oplus 3}
\xrightarrow{\Phi}
\Ocal_B^{\oplus 3}
\xrightarrow{
\begin{pmatrix}
f_3 & f_2 & f_1
\end{pmatrix}}
(f_1,f_2,f_3)
\rightarrow 0,
$$
where
$$
\Phi=
\begin{pmatrix}
f_2&-f_1&0\\
-f_3&0&f_1\\
0&f_3&-f_2
\end{pmatrix}.
$$This gives a locally free resolution of $\Ccal$:
$$
0\rightarrow \Ocal_B
\xrightarrow{
\begin{pmatrix}
f_1\\
f_2\\
f_3\\
0
\end{pmatrix}}
\Ocal_B^{\oplus 3}\oplus \Ocal_B
\xrightarrow{\Psi}
\Ocal_B^{\oplus 3}\oplus \Ocal_B^{\oplus 2}
\rightarrow \Ccal
\rightarrow 0,
$$
where
$$
\Psi=
\begin{pmatrix}
f_2&-f_1&0&f_3\\
-f_3&0&f_1&0\\
0&f_3&-f_2&0\\
0&0&0&f_2\\
0&0&0&f_1
\end{pmatrix}.
$$

Viewing the functions $f_i$ as pulled back to $X=B\times\Delta$, we now
deform the first map by introducing the parameter $t$:
$$
\begin{pmatrix}
f_1\\
f_2\\
f_3\\
t
\end{pmatrix}.
$$
A direct computation gives the following exact complex on $X$:
$$
0\rightarrow \Ocal
\xrightarrow{
\begin{pmatrix}
f_1\\
f_2\\
f_3\\
t
\end{pmatrix}}
\Ocal^{\oplus 3}\oplus \Ocal
\xrightarrow{M}
\Ocal^{\oplus 3}\oplus \Ocal^{\oplus 2}
\rightarrow \Ecal
\rightarrow 0,
$$
where
$$
M=
\begin{pmatrix}
f_2&-f_1&-t&f_3\\
-f_3&0&f_1&0\\
0&f_3&-f_2&0\\
0&-t&0&f_2\\
-t&0&0&f_1
\end{pmatrix}.
$$

The main point is the following.

\begin{prop}\label{SmoohtlizeAClassOfSingularities}
The family $\Ecal$ has locally free general fibers, and its central fiber
$\Ccal$ has an isolated point singularity at the origin. Moreover,
$\Ccal$ fits into
$$
0\rightarrow \Ocal_B
\xrightarrow{
\begin{pmatrix}
f_1\\
f_2\\
f_3^2
\end{pmatrix}}
\Ocal_B^{\oplus 2}\oplus (f_1,f_2,f_3)
\rightarrow \Ccal
\rightarrow 0.
$$
\end{prop}

\begin{proof}
We first show that $M$ has rank $3$ on $X\setminus\{(o,0)\}$. The vector
$(f_1,f_2,f_3,t)^T$ lies in $\Kcal er(M)$, so the rank is at most $3$.
For $t=0$, the claim follows from the resolution of $\Ccal$ above. It
remains to consider the case $t\neq 0$. The minor
$$
\begin{pmatrix}
f_2&-f_1&-t\\
0&-t&0\\
-t&0&0
\end{pmatrix}
$$
has determinant $t^3$, hence is nonzero. Thus $M$ has rank $3$ away
from $(o,0)$.

It follows that the cokernel $\Ecal$ is locally free away from $(o,0)$.
Specializing the displayed complex at $t=0$ recovers the resolution of
$\Ccal$ above, so the central fiber of $\Ecal$ is $\Ccal$. Since $\Ccal$
is the elementary modification of the torsion-free sheaf $\Fcal$ described
above, it is torsion-free and has an isolated singularity at the origin.
\end{proof}

As a direct corollary, a fixed reflexive singularity may admit more than
one smoothable elementary modification.

\begin{cor}\label{Cor-nonuniqueness}
Let $\Fcal$ be the rank-two reflexive sheaf on $B$ defined by
$$
0\rightarrow \Ocal_B
\xrightarrow{
\begin{pmatrix}
z_1\\
z_2^4\\
z_3^2
\end{pmatrix}}
\Ocal_B^{\oplus 3}
\rightarrow \Fcal
\rightarrow 0.
$$
Then $\Fcal$ admits at least two nonisomorphic smoothable elementary
modifications at the origin.
\end{cor}

\begin{proof}
There are two inequivalent ways to realize the same reflexive sheaf in
the form above. One may take
$$
        (f_1,f_2,f_3)=(z_1,z_2^4,z_3),
$$
or, after swapping the second and third summands in the presentation of
$\Fcal$,
$$
        (f_1,f_2,f_3)=(z_1,z_3^2,z_2^2).
$$
The construction gives smoothable elementary modifications
$\Ccal_1,\Ccal_2\subset \Fcal$. Their finite-length quotients are
$$
        \Fcal/\Ccal_1\cong \Ocal_B/(z_1,z_2^4,z_3),
        \qquad
        \Fcal/\Ccal_2\cong \Ocal_B/(z_1,z_3^2,z_2^2).
$$
These quotients are not isomorphic. In particular, $\Ccal_1$ and $\Ccal_2$ are not isomorphic. 
\end{proof}
\subsection{The minimal isolated singularity}
\label{subsec:minimal-singularity}

We now specialize the construction to the minimal case, i.e. $m^{\mathrm{alg}}=1$. Let $\Fcal_{\min}$ be the rank-two reflexive sheaf on $B$ defined by
$$
0\rightarrow \Ocal_B
\xrightarrow{
\begin{pmatrix}
z_1\\
z_2\\
z_3^2
\end{pmatrix}}
\Ocal_B^{\oplus 3}
\rightarrow \Fcal_{\min}
\rightarrow 0 .
$$
Let $e_1,e_2,e_3$ be the images in $\Fcal_{\min}$ of the three standard
generators of $\Ocal_B^{\oplus 3}$. Define
$$
        q_{\min}:\Fcal_{\min}\rightarrow \Ocal_B/(z_1,z_2,z_3)
$$
by
$$
        q_{\min}(e_1)=0,\qquad q_{\min}(e_2)=0,\qquad q_{\min}(e_3)=1,
$$
and set
$$
        \Ccal_{\min}:=\Kcal er(q_{\min}).
$$

\begin{exam}\label{SimplestFertileFamily}
The family produced by the construction above can be written explicitly as
$$
0\rightarrow \Ocal
\xrightarrow{
\begin{pmatrix}
z_1\\
z_2\\
z_3\\
t
\end{pmatrix}}
\Ocal^{\oplus 4}
\xrightarrow{\Phi}
\Ocal^{\oplus 5}
\rightarrow \Ecal_{\min}
\rightarrow 0,
$$
where
$$
M=
\begin{pmatrix}
z_2&-z_1&-t&z_3\\
-z_3&0&z_1&0\\
0&z_3&-z_2&0\\
0&-t&0&z_2\\
-t&0&0&z_1
\end{pmatrix}.
$$
Then $\Ecal_{\min}$ degenerates to an isolated point singularity at the
origin, and its central fiber is $\Ccal_{\min}$. In particular,
$$
        \Ccal_{\min}^{**}\cong \Fcal_{\min},
        \qquad
        \Ccal_{\min}^{**}/\Ccal_{\min}
        \cong \Ocal_B/(z_1,z_2,z_3) .
$$
\end{exam}

The next result shows that, in the minimal case, no other reflexive
singularity can occur.

\begin{lem}\label{lem:minimal-reflexive-hull}
Suppose $\Ecal$ is a family degenerating to an isolated point singularity
at the origin with $m^{\mathrm{alg}}=1$. Then, after a local analytic
change of coordinates at the origin,
$$
        \Ccal^{**}\cong \Fcal_{\min}.
$$
\end{lem}

\begin{proof}
By Theorem~\ref{thm:bubbling-rigidity},
$l(\Ecal xt^1(\Ccal,\Ocal_B))=2$. Set $\Gcal=\Ccal^{**}$. By
Lemma~\ref{Lemma-Defect only depends on reflexive hull},
$l(\Ecal xt^1(\Gcal,\Ocal_B))=2$.

Since $\Gcal$ is reflexive but not locally free at the origin, it has a
minimal presentation near the origin
$$
0\rightarrow \Ocal_B^{\oplus r}
\xrightarrow{\Phi}
\Ocal_B^{\oplus r+2}
\rightarrow \Gcal
\rightarrow 0 .
$$
Here $r\geq 1$, and minimality means that all entries of $\Phi$ lie in
$(z_1,z_2,z_3)$. Dualizing gives
$$
        \Ecal xt^1(\Gcal,\Ocal_B)
        \cong
        \Ocal_B^{\oplus r}/\Img(\Phi^t).
$$
Since $\Img(\Phi^t)\subset (z_1,z_2,z_3)\Ocal_B^{\oplus r}$ and this
quotient has length two, one has $1\leq r\leq 2$.

\begin{clm}\label{clm:r-one}
$r=1$.
\end{clm}

Assuming the claim, write
$$
        \Phi=
        \begin{pmatrix}
        f_1\\
        f_2\\
        f_3
        \end{pmatrix}.
$$
Then $\Ocal_B/(f_1,f_2,f_3)$ has length two. Let
$I=(f_1,f_2,f_3)$. Since $\Ocal_B/I$ is an Artinian local ring of length
two, $(z_1,z_2,z_3)^2\subset I$. Moreover, the image of $I$ in
$(z_1,z_2,z_3)/(z_1,z_2,z_3)^2$ has dimension two. After a linear
change of coordinates, we may assume that this image is spanned by the
classes of $z_1$ and $z_2$. Hence
$$
        I=(z_1,z_2)+(z_1,z_2,z_3)^2=(z_1,z_2,z_3^2).
$$
After changing the basis of $\Ocal_B^{\oplus 3}$, the presentation of
$\Gcal$ becomes
$$
0\rightarrow \Ocal_B
\xrightarrow{
\begin{pmatrix}
z_1\\
z_2\\
z_3^2
\end{pmatrix}}
\Ocal_B^{\oplus 3}
\rightarrow \Gcal
\rightarrow 0 .
$$
Thus $\Gcal\cong \Fcal_{\min}$, proving the lemma up to the claim.
\end{proof}

\begin{proof}[Proof of the claim]
Suppose $r=2$. Since
$\Ocal_B^{\oplus 2}/\Img(\Phi^t)$ has length two and
$\Img(\Phi^t)\subset (z_1,z_2,z_3)\Ocal_B^{\oplus 2}$, one must have
$$
        \Img(\Phi^t)=(z_1,z_2,z_3)\Ocal_B^{\oplus 2}.
$$
Thus $(z_1,z_2,z_3)\Ocal_B^{\oplus 2}$ is generated by the four columns
of $\Phi^t$. This contradicts Nakayama's lemma, since
$$
        \dim_{\CBbb}
        (z_1,z_2,z_3)\Ocal_B^{\oplus 2}/
        (z_1,z_2,z_3)^2\Ocal_B^{\oplus 2}
        =6 .
$$
Hence $r=1$.
\end{proof}

We now identify the central fiber itself, not only its reflexive hull. Let
$$
        q:\Fcal_{\min}\to \Ocal_B/(z_1,z_2,z_3)
$$
be a length-one quotient, and write
$$
        q(e_1)=a,\qquad q(e_2)=b,\qquad q(e_3)=c .
$$

\begin{lem}\label{lem:nonstandard-length-one-quotient}
Suppose $(a,b)\neq (0,0)$, then $\Kcal er(q)$ is not smoothable.
\end{lem}

\begin{proof}
After a linear change of the first two generators, together with the
corresponding linear change of $z_1,z_2$, we may assume $a=1$ and $b=0$.
Replacing $e_3$ by $e_3-c e_1$ and replacing $z_1$ by $z_1+c z_3^2$, we
may further assume $c=0$. Thus it is enough to consider the quotient
$$
        q(e_1)=1,\qquad q(e_2)=q(e_3)=0.
$$
Set $\Gcal=\Kcal er(q)$. Then $\Gcal$ is generated by
$$
        z_2e_1,\qquad z_3e_1,\qquad e_2,\qquad e_3 .
$$
Let $\Pcal$ be the kernel of the induced surjection
$$
        \Ocal_B^{\oplus 4}\to \Gcal .
$$
Then $\Pcal$ is a rank-two reflexive sheaf with an isolated singularity at
the origin, and
$$
        \Ecal xt^1(\Pcal,\Ocal_B)\cong \Ecal xt^2(\Gcal, \Ocal_B)\cong \Ocal_B/(z_1,z_2,z_3) .
$$

Suppose that $\Gcal$ were smoothable, and let $\Ecal$ be a smoothing.
Lifting the four generators of $\Gcal$ to local sections of $\Ecal$ gives,
by Nakayama's lemma, a surjection
$$
        \Ocal^{\oplus 4}\to \Ecal
$$
near $(o,0)$. Let $\Hcal$ be its kernel. Since $\Ecal$ is flat over
$\Delta$, restriction to the central fiber gives
$$
0\rightarrow \left.\Hcal\right|_B
\rightarrow \Ocal_B^{\oplus 4}
\rightarrow \Gcal
\rightarrow 0 .
$$
Hence $\left.\Hcal\right|_B\cong \Pcal$. Moreover, $\Hcal_t$ is locally
free of rank two for $t\neq 0$, and $\Hcal$ is locally free away from
$(o,0)$. Thus $\Pcal$ would be smoothable. This contradicts
Corollary~\ref{OddParityObstruction}, since
$l(\Ecal xt^1(\Pcal,\Ocal_B))=1$. Therefore $\Gcal$ is not smoothable.
\end{proof}

\begin{lem}[Uniqueness of the smoothable elementary modification]
\label{lem:unique-smoothable-modification}
If $\Gcal\subset \Fcal_{\min}$ is a smoothable elementary modification at the
origin whose quotient has finite length, then $\Gcal=\Ccal_{\min}$.
\end{lem}

\begin{proof}
Let $0\to \Gcal\to \Fcal_{\min}\to \tau\to 0$ be such an elementary
modification. Since $\tau$ has finite length on the smooth threefold
$B$, one has $\Ecal xt^i(\tau,\Ocal_B)=0$ for $i<3$. Hence
$$
        \Ecal xt^1(\Gcal,\Ocal_B)
        \cong
        \Ecal xt^1(\Fcal_{\min},\Ocal_B)
        \cong
        \Ocal_B/(z_1,z_2,z_3^2).
$$
Thus $l(\Ecal xt^1(\Gcal,\Ocal_B))=2$. If $\Gcal$ is smoothable, then
$\Gcal^{**}\cong \Fcal_{\min}$, and Theorem~\ref{thm:bubbling-rigidity}
forces
$$
        l(\Fcal_{\min}/\Gcal)=1 .
$$

Therefore $\Gcal$ is the kernel of a length-one quotient
$q:\Fcal_{\min}\to \Ocal_B/(z_1,z_2,z_3)$. Write
$$
        q(e_1)=a,\qquad q(e_2)=b,\qquad q(e_3)=c .
$$
If $(a,b)\neq (o,0)$, then
Lemma~\ref{lem:nonstandard-length-one-quotient} shows that
$\Kcal er(q)$ is not smoothable. Hence a smoothable quotient must satisfy
$a=b=0$, so it is proportional to $q_{\min}$. Therefore
$$
        \Gcal=\Kcal er(q_{\min})=\Ccal_{\min}.
$$
\end{proof}

\begin{cor}[Uniqueness in the minimal bubbling case]
\label{cor:unique-minimal-central-fiber}
Let $\Ecal$ be a family whose central fiber $\Ccal$ has an isolated point
singularity at the origin and satisfies $m^{\mathrm{alg}}=1$. Then, after a local analytic
change of coordinates at the origin, the pair
$$
        \Ccal\subset \Ccal^{**}
$$
is isomorphic to
$$
        \Ccal_{\min}\subset \Fcal_{\min}.
$$
\end{cor}

\begin{proof}
By Lemma~\ref{lem:minimal-reflexive-hull}, after a local analytic change
of coordinates, 
$$
\Ccal^{**}\cong \Fcal_{\min}.
$$ 
Since
$m^{\mathrm{alg}}=1$, the inclusion $\Ccal\subset \Ccal^{**}$ is a
length-one elementary modification. Lemma~\ref{lem:unique-smoothable-modification}
identifies the only smoothable one.
\end{proof}

\begin{rmk}\label{Rmk:Nonreducedness}
The corollary classifies the central fiber as an elementary modification
of its reflexive hull. It does not classify the smoothing family over
$\Delta$. Indeed, if $\Ecal$ is such a family, then for any $k\geq 2$
the ramified base change $p_k^*\Ecal$, where
$$
p_k(z_1,z_2,z_3,t)=(z_1,z_2,z_3,t^k),
$$
is again a family of the same type. In general, these base changes need
not be isomorphic as families over $\Delta$.

For the minimal family considered above, the resulting family
$\Ecal_k=p_k^*\Ecal$ satisfies
$$
\Ecal xt^1(\Ecal_k,\Ocal)
\cong
\Ocal/(z_1,z_2,z_3,t^k).
$$
Thus this Ext sheaf is supported on the $k$-th infinitesimal thickening
of the central point and is not annihilated by $t$. In particular,
$\Ecal_k$ is not reduced in the sense of Perrin \cite{Perrin:2001}.
\end{rmk}

\section{Global appearance of the minimal singularity and bubbling}\label{Section-Global appearance of the minimal singularity and bubbling}
In this section we globalize the minimal model from
Section~\ref{subsec:minimal-singularity}.  We write
$\Ocal_{\PBbb^3}(k)$ simply as $\Ocal(k)$, and use the same notation for
its pullback to $\PBbb^3\times \Delta$ when the parameter $t$ is present.
The local construction is the same as before; the new point is that the resulting singularities arise as limits of stable rank-two bundles in a
global family.  We then examine the associated
Hermitian--Yang--Mills degeneration.

\subsection{Main results}
\label{subsec:minimal-global-conclusions}
The main output of this section is the following global realization of the
minimal local model. 
\begin{thm}[Minimal global degeneration and analytic bubbles]
\label{thm:minimal-global-degeneration}
There exists a flat family
$\Ecal$ of rank-two sheaves on $\PBbb^3\times\Delta$ with central fiber
$\Ccal$ such that the following hold.

\begin{enumerate}
\item For every $t\neq 0$, the fiber $\Ecal_t$ is stable and locally free.
Moreover,
$$
\Ccal^{**}/\Ccal
\cong
\Ocal_{\PBbb^3}/\Ical,
$$
where $\Ical$ is the ideal sheaf associated to two distinct points $p_1,p_2$. Furthermore, near each $p_j$, up to local coordinate change, the pair
$$
\Ccal\subset \Ccal^{**}
$$
is locally isomorphic to the minimal local pair
$$
\Ccal_{\min}\subset \Fcal_{\min}.
$$

\item Let $A_i$ be the HYM connection on $\Ecal_{t_i}$ for a sequence
$t_i\to 0$. After passing to a subsequence, the Uhlenbeck limit
$$
(A_\infty,\Ecal_\infty,Z_b)
$$
satisfies
$$
\Ecal_\infty\cong \Ccal^{**},
\qquad
Z_b=0,
\qquad
\operatorname{Sing}(A_\infty)=\{p_1,p_2\}.
$$
In particular,
$$
|F_{A_i}|^2\,\dvol
\rightharpoonup
|F_{A_\infty}|^2\,\dvol
$$
as Radon measures on $\PBbb^3$.

\item For each $j=1,2$, the analytic tangent-cone data of $A_\infty$ at
$p_j$ are
$$
\mathcal T_{p_j}(A_\infty)=(A_{T,j},[L_j]),
$$
where $A_{T,j}$ is flat and
$L_j\subset \CBbb^3$ is a complex line. Furthermore,
$$
\Theta_{p_j}(\{A_i\}_i)
=
\Theta_{p_j}(A_\infty)
=
1 .
$$

\item Fix one of the two points $p_j$. After passing to a further subsequence,
there are centered scales $r_i\downarrow 0$ such that, in holomorphic
coordinates centered at $p_j$, the rescaled connections
$$
\widetilde A_i:=\delta_i^*A_i,
\qquad
\delta_i(z)=r_i z,
$$
have a Uhlenbeck limit
$$
(B_\infty,\mathcal B_\infty,Z_b')
$$
on $\CBbb^3$ satisfying
$$
Z_b'=0.
$$
The limiting sheaf $\mathcal B_\infty$ is locally free, and $B_\infty$ is
non-flat. Moreover,
$$
\Theta_\infty(B_\infty)
:=
\lim_{R\to\infty}
\frac{1}{8\pi^3R^2}
\int_{B_R(0)} |F_{B_\infty}|^2\,\dvol
=
1,
$$
and the analytic tangent-cone data of $B_\infty$ at infinity are
$$
\mathcal T_\infty(B_\infty)=(B_T,[L]),
$$
where $B_T$ is flat and $L\subset\CBbb^3$ is a complex line.
\end{enumerate}
\end{thm}

We call the connection $B_\infty$ in item $(4)$ a
\emph{minimal-density analytic bubble} associated with the global
degeneration. With our normalization,
$$
\Theta_\infty(B_\infty)=1,
$$
and a multiplicity-one complex line contributes exactly one unit of density.
Thus $B_\infty$ is a smooth non-flat HYM connection on $\CBbb^3$ with
minimal quadratic energy growth compatible with a nonzero tangent-cone
blow-up cycle.

The point is that the non-flatness of $B_\infty$ is not visible in the
connection part of its tangent cone at infinity.  Indeed,
$$
\mathcal T_\infty(B_\infty)=(B_T,[L]),
\qquad
B_T\ \text{flat},
$$
where $L\subset \CBbb^3$ is a complex line. The limiting cone connection is
flat, and the remaining curvature concentration is recorded only by the
multiplicity-one line $[L]$. In this sense, the global degeneration gives a
theoretic construction of a minimal nontrivial smooth HYM bubble on
$\CBbb^3$ whose tangent cone at infinity has flat connection part.

In Subsection~\ref{subsec:compactification-minimal-bubble}, we discuss a
related compactification question. Assuming the expected
compatibility between algebraic compactifications and tangent-cone data at
infinity, we show that the boundary sheaf at infinity is forced into one of
two explicit types: the unique non-split extension
$$
0\longrightarrow \Ocal_{H_\infty}
\longrightarrow \Gcal_\infty
\longrightarrow \Ical_q
\longrightarrow 0,
$$
or the split sheaf
$$
\Ocal_{H_\infty}\oplus \Ical_q .
$$
The split type is excluded for algebraic bubbles coming from families
forming the minimal singularity. Assuming the folklore expectation that the analytic bubble has split type at infinity, this gives a precise sense in which the centered analytic bubble differs from the algebraic bubble obtained by blowing up the degenerating family.

\subsection{Uhlenbeck compactness}\label{Section-Uhlenbeck compactness}
In this section, we recall known results for
Hermitian--Yang--Mills connections that will be used below and refer readers to
\cite{UhlenbeckPreprint,Nakajima:88,Tian:00,BandoSiu:94,ChenWentworth:21a, ChenSun:20a,ChenSun:18, ChenSun:20b, ChenSun:19} for details. We tailor the discussion to the case of rank-two bundles on threefolds.

Let $X$ be a K\"ahler threefold, and let $A_i$ be a sequence of HYM
connections on a fixed smooth Hermitian vector bundle $(E,H)\to X$, with
uniformly bounded Yang--Mills energy. After passing to a subsequence, there exist a subvariety
$\Sigma\subset X$ of complex codimension at least two, a smooth Hermitian
vector bundle $(E_\infty,H_\infty)$ over $X\setminus\Sigma$, and an
admissible HYM connection $A_\infty$ on $E_\infty$, with the following
property. For $i\gg 1$, there are Hermitian bundle isomorphisms
$$
P_i:(E,H)|_{X\setminus\Sigma}
\longrightarrow
(E_\infty,H_\infty)
$$
such that
$$
(P_i^{-1})^*A_i\longrightarrow A_\infty
$$
smoothly on compact subsets of $X\setminus\Sigma$ (\cite{Tian:00, UhlenbeckPreprint, Nakajima:88, BandoSiu:94, ChenWentworth:21a}). Here by (\cite{BandoSiu:94}), the holomorphic bundle
$(E_\infty,\bar\partial_{A_\infty})$ extends uniquely across codimension at
least two to a reflexive sheaf $\Ecal_\infty$ on $X$. Moreover, after gauge transforms,
$A_\infty$ extends smoothly across the locus where $\Ecal_\infty$ is
locally free. We write
$$
\operatorname{Sing}(A_\infty)
=
\operatorname{Sing}(\Ecal_\infty)
$$
for the essential singular set of the admissible limit. Since $X$ has
complex dimension three, this set is finite. Then
$$
\Sigma=\Sigma_b\cup \operatorname{Sing}(A_\infty),
$$
where $\Sigma_b$ is the closure of the pure complex codimension-two part of
$\Sigma\setminus\operatorname{Sing}(A_\infty)$. Thus, on a threefold,
$\Sigma_b$ is a complex curve, possibly empty. $\Sigma_b$ is usually referred as the
blow-up locus. If
$$
\Sigma_b=\bigcup_k\Sigma_k
$$
is the decomposition into irreducible components, after passing to a further subsequence, 
$$
\frac{1}{8\pi^2}|F_{A_i}|^2\,\dvol
\rightharpoonup
\frac{1}{8\pi^2}|F_{A_\infty}|^2\,\dvol
+
\sum_k m_k[\Sigma_k]
$$
as Radon measures on $X$ where $m_k$ is a positive integer called analytic multiplicity associated to the irreducible pure codimension two component $\Sigma_k$.  The blow-up cycle of the sequence is defined by
$$
Z_{b}:=\sum_k m_k\Sigma_k .
$$ 
Below, we call
$$
(A_\infty, \Ecal_\infty, Z_b)
$$
the Uhlenbeck  limit of the sequence $A_i$ and sometimes we might omit $\Ecal_\infty$ by simply writing it as $(A_\infty, Z_b)$.  For a point $x\in \operatorname{Sing}(\Ecal_\infty)$, we write
$$
\mathcal T_x(A_\infty)=(A_{T,x},Z_{T,x})
$$
for the analytic tangent-cone data of $A_\infty$ at $x$. More precisely, these data are obtained as Uhlenbeck limits of rescaled sequences centered at $x$ (\cite{Tian:00}). Here $A_{T,x}$ is
an admissible HYM cone connection, and $Z_{T,x}$ is the tangent-cone
blow-up cycle, viewed as a conical codimension-two cycle in
$\CBbb^3$. When needed, we denote by $\Ecal_{T,x}$ the
underlying reflexive cone sheaf of $A_{T,x}$. By \cite{ChenSun:20a,ChenSun:18,ChenSun:20b,ChenSun:19}, these data are intrinsic algebraic invariants of the stalk of $\Ecal_\infty$ at $x$. Moreover, they can be computed from the optimal extension of $\Ecal_\infty$ at $x$ (\cite{ChenSun:20b}). Since $\Ecal_\infty$ has rank two, this optimal extension is necessarily semistable.

\subsection{Global appearance of the minimal model}
Let
$$
        p_1=[0,0,0,1],\qquad p_2=[0,1,0,0],
$$
and let
$
        \Ical=(Z_1,Z_2Z_4,Z_3)\subset \Ocal_{\PBbb^3}.
$
be the ideal sheaf associated to $p_1,p_2.$ Consider the rank-two reflexive sheaf $\Gcal$ on $\PBbb^3$ defined by
$$
0\rightarrow \Ocal(-3)
\xrightarrow{
\begin{pmatrix}
Z_1\\
Z_2Z_4\\
Z_3^2
\end{pmatrix}}
\Ocal(-2)\oplus \Ocal(-1)^{\oplus 2}
\rightarrow \Gcal
\rightarrow 0 .
$$
The sheaf $\Gcal$ has precisely two isolated singularities, at $p_1$ and
$p_2$.  At both points, its local model is the minimal reflexive
singularity $\Fcal_{\min}$ from Section~\ref{subsec:minimal-singularity}.
Moreover, since $H^0(\PBbb^3,\Gcal)=0$ and $c_1(\Gcal)=-1$, the sheaf
$\Gcal$ is stable.

Set $\Fcal:=\Gcal(1).$ Let
$$
        q:\Fcal\rightarrow \Ocal_{\PBbb^3}/\Ical
$$
be the quotient induced by projection of
$
        \Ocal(-1)\oplus \Ocal^{\oplus 2}
$
onto the third factor, followed by the quotient map
$\Ocal_{\PBbb^3}\to \Ocal_{\PBbb^3}/\Ical$.  Define
$$
        \Ccal:=\Kcal er(q).
$$
Then
$$
        \Ccal^{**}\cong \Fcal,\qquad
        \Ccal^{**}/\Ccal\cong \Ocal_{\PBbb^3}/\Ical .
$$
Thus, near each of $p_1$ and $p_2$, the pair
$$
        \Ccal\subset \Ccal^{**}
$$
is isomorphic to the minimal pair
$$
        \Ccal_{\min}\subset \Fcal_{\min}.
$$
In particular, the algebraic bubbling multiplicity is one at each
singular point.

We now apply the same construction as in
Section~\ref{DeformingSingularities}.  Let $P$ be a homogeneous polynomial
of degree two, and introduce a parameter $t$. Set
$$
\begin{aligned}
\mathcal U
&:=\Ocal(-3)\oplus \Ocal(-2)\oplus \Ocal(-3)\oplus \Ocal(-2),\\
\mathcal V
&:=\Ocal(-2)\oplus \Ocal(-1)^{\oplus 2}\oplus \Ocal\oplus \Ocal(-1),
\end{aligned}
$$
and let
$$
        s=
        \begin{pmatrix}
        Z_1\\
        Z_2Z_4\\
        Z_3\\
        tP
        \end{pmatrix}.
$$
Define $\Ecal$ on $\PBbb^3\times\Delta$ by the complex
$$
        0\rightarrow \Ocal(-4)
        \xrightarrow{s}
        \mathcal U
        \xrightarrow{M}
        \mathcal V
        \rightarrow \Ecal
        \rightarrow 0,
$$
where, with respect to the above decompositions,
$$
M=
\begin{pmatrix}
Z_3&0&-Z_1&0\\
0&-Z_3&Z_2Z_4&0\\
-Z_2Z_4&Z_1&-tP&Z_3\\
0&-tP&0&Z_2Z_4\\
-tP&0&0&Z_1
\end{pmatrix}.
$$
For $t=0$, this complex specializes to the resolution of
$\Ccal$.

\begin{prop}\label{prop:global-minimal-family}
Let $P$ be a homogeneous polynomial of degree two such that
$P(p_1)\neq 0$ and $P(p_2)\neq 0$.  Then $\Ecal_t$ is a stable
rank-two bundle on $\PBbb^3$ for every $t\neq 0$, and the central fiber
of $\Ecal$ is $\Ccal=\Kcal er(q)$ being also stable.  Its only singularities are $p_1$ and $p_2$, and
at each point the local model is the minimal pair
$$
        \Ccal_{\min}\subset \Fcal_{\min}.
$$
\end{prop}

\begin{proof}
The description of the central fiber follows by setting $t=0$ in the
displayed complex.  The local analysis is the same as in
Proposition~\ref{SmoohtlizeAClassOfSingularities}.  Away from $p_1$ and
$p_2$, the complex has the expected rank.  Near each $p_i$, the condition
$P(p_i)\neq 0$ identifies the family with the local smoothing of the
minimal model.  Hence $\Ecal_t$ is locally free for $t\neq 0$, while the
central fiber has precisely the two minimal point singularities described
above.

It remains to check stability.  Since $c_1(\Ecal_t)=1$, it is enough to
show $H^0(\PBbb^3,\Ecal_t(-1))=0$.  Let
$$
        \Kcal_t:=\Img(M_t).
$$
From the displayed resolution and the vanishing of intermediate cohomology
of line bundles on $\PBbb^3$, one obtains
$$
        H^1(\PBbb^3,\Kcal_t(-1))=0.
$$
The exact sequence
$$
        0\to \Kcal_t(-1)
        \to
        \Ocal(-3)\oplus \Ocal(-2)^{\oplus 2}\oplus \Ocal(-1)\oplus \Ocal(-2)
        \to \Ecal_t(-1)\to 0
$$
therefore gives $H^0(\PBbb^3,\Ecal_t(-1))=0$.  Thus $\Ecal_t$ is stable.
\end{proof}

\subsection{Uhlenbeck limits and tangent cones for the global degeneration}
By the Donaldson--Uhlenbeck--Yau theorem
(\cite{Donaldson:87a,UhlenbeckYau:86}), there exists a HYM metric
$H_t$ on $\Ecal_t$ for every $t\neq 0$. Let $A_t$ denote the associated
Chern connection. Since the fibers $\Ecal_t$ form a smooth family of
locally free sheaves over the punctured disc, their underlying smooth
complex vector bundles are all isomorphic.

Choose Hermitian identifications and regard the $A_t$ as unitary
connections on a fixed smooth Hermitian vector bundle $(E,H)$. Fix a
sequence $t_i\to 0$ and set $A_i:=A_{t_i}$. Since the Chern classes are
constant in the family, the Yang--Mills energies of the HYM connections
$A_i$ are uniformly bounded.
By Uhlenbeck compactness (\ref{Section-Uhlenbeck compactness}), after passing to a
subsequence we obtain an Uhlenbeck limit
$$
(A_\infty,\Ecal_\infty,Z_{b}).
$$
Here $A_\infty$ is an admissible HYM connection on the limiting reflexive
sheaf $\Ecal_\infty$, and $Z_{b}$ is the codimension-two defect cycle. 
For a point $x\in \operatorname{Sing}(\Ecal_\infty)$, we write
$$
\mathcal T_x(A_\infty)=(A_{T,x},Z_{T,x})
$$
for the analytic tangent-cone data of $A_\infty$ at $x$.
 
\begin{prop}\label{StableReflexiveSheafImpliesHYMConvergence}
The Uhlenbeck limit
$$
(A_\infty,\Ecal_\infty,Z_b)
$$
satisfies the following.
\begin{enumerate}
\item The limiting reflexive sheaf is
$$
\Ecal_\infty\cong \Ccal^{**},
$$
and there is no pure codimension two bubbling:
$$
Z_b=0, \quad \operatorname{Sing}(\Ecal_\infty)=\{p_1,p_2\}.
$$
Consequently, as Radon measures on $\PBbb^3$,
$$
|F_{A_i}|^2\,\dvol
\rightharpoonup
|F_{A_\infty}|^2\,\dvol .
$$

\item For each $j=1,2$, the analytic tangent-cone data
$$
\mathcal T_{p_j}(A_\infty)=(A_{T,j},Z_{T,j})
$$
satisfy
$$
Z_{T,j}=[L_j],
$$
where $A_{T,j}$ is flat and
$L_j\subset T_{p_j}\PBbb^3\cong \CBbb^3$ is a complex line passing through the origin.

\item For each $j=1,2$, the energy densities are equal 
$$
\Theta_{p_j}(\{A_i\}_i)=\Theta_{p_j}(A_\infty)=1 .
$$
\end{enumerate}
\end{prop}
\begin{proof}
Denote
$$
\mu_i:=\frac{1}{8\pi^2}|F_{A_i}|^2\,\dvol,\qquad
\mu_\infty:=\frac{1}{8\pi^2}|F_{A_\infty}|^2\,\dvol.
$$

(1) follows from the general result in \cite{GSTW:18}. We include a simplified discussion in our setting. We first identify the limiting reflexive sheaf. By ``semi-continuity" (\cite{Donaldson:85}), there is a nonzero holomorphic map
$$
f:\Ccal^{**}\longrightarrow \Ecal_\infty .
$$
The sheaf $\Ecal_\infty$ is polystable of the same slope as $\Ccal^{**}$,
while $\Ccal^{**}$ is stable. Thus $f$ has to be an isomorphism. In particular,
$$
\operatorname{Sing}(\Ecal_\infty)
=
\operatorname{Sing}(\Ccal^{**})
=
\{p_1,p_2\}.
$$
It remains to show that there is no pure codimension-two bubbling. Write
$$
Z_{b}=\sum_k m_k\Sigma_k.
$$
The singular Bott--Chern formula
of \cite{SibleyWentworth:15} identifies the cohomology class of
$Z_{b}$ with the codimension-two Chern-character defect between
the smooth fibers and the limiting reflexive sheaf. In the present case this
defect class vanishes. Indeed, since
$$
0\longrightarrow \Ccal\longrightarrow \Ccal^{**}\longrightarrow Q
\longrightarrow 0
$$
has $Q$ supported in dimension zero, we have
$$
\operatorname{ch}_2(\Ccal^{**})
=
\operatorname{ch}_2(\Ccal).
$$
On the other hand, flatness of the family gives
$$
\operatorname{ch}_2(\Ecal_t)
=
\operatorname{ch}_2(\Ccal)
\qquad (t\neq 0).
$$
Together with $\Ecal_\infty\cong \Ccal^{**}$, this shows that blow-up cycle is zero. Since $Z_{b}$ is an effective curve cycle on $\PBbb^3$, and every
irreducible curve in $\PBbb^3$ has positive degree, a zero cohomology class
forces $Z_{b}=0.$ In particular, $\mu_i\rightharpoonup \mu_\infty$ as Radon measures on $\PBbb^3$. This proves the first assertion.

We now compute the tangent-cone data at the two point singularities.
After choosing analytic coordinates centered at either $p_1$ or $p_2$,
the sheaf $\Ccal^{**}$ is locally isomorphic to the cokernel of
$$
0\longrightarrow \Ocal
\xrightarrow{
\begin{pmatrix}
z_1\\
z_2\\
z_3^2
\end{pmatrix}}
\Ocal^3
\longrightarrow \Fcal
\longrightarrow 0 .
$$

Fix $j\in\{1,2\}$. This is the $k=2$ case of
\cite[Section~5]{ChenSun:18}. To make the algebraic input explicit, blow up
$p_j$ and denote the exceptional divisor by $D_j\cong \PBbb^2.$ The homogeneous presentation above gives an algebraic tangent-cone
extension whose restriction to $D_j$ is the rank-two bundle $\Gcal_j$
defined by
$$
0\longrightarrow \Ocal_{\PBbb^2}
\xrightarrow{
\begin{pmatrix}
Z_1\\
Z_2\\
Z_3^2
\end{pmatrix}}
\Ocal_{\PBbb^2}(1)\oplus \Ocal_{\PBbb^2}(1)\oplus
\Ocal_{\PBbb^2}(2)
\longrightarrow \Gcal_j
\longrightarrow 0 .
$$
The quotient is locally free because the sections
$Z_1,Z_2,Z_3^2$ have no common zero on $\PBbb^2$. A direct computation of
the Harder--Narasimhan--Seshadri graded object gives
$$
\operatorname{Gr}^{\mathrm{HNS}}(\Gcal_j)
\cong
\Ocal_{\PBbb^2}(2)
\oplus
\left(\mathcal I_{q_j}\otimes \Ocal_{\PBbb^2}(2)\right),
\qquad
q_j=[0:0:1].
$$
Hence
$$
\left(\operatorname{Gr}^{\mathrm{HNS}}(\Gcal_j)\right)^{**}
\cong
\Ocal_{\PBbb^2}(2)^{\oplus 2},
$$
and
$$
\left(\operatorname{Gr}^{\mathrm{HNS}}(\Gcal_j)\right)^{**}
/
\operatorname{Gr}^{\mathrm{HNS}}(\Gcal_j)
\cong
\Ocal_{q_j}
$$
has length one. The
analytic tangent-cone data are determined by this double dual and by the
codimension-two cycle of the quotient above. The cone connection
corresponding to
$\Ocal_{\PBbb^2}(2)^{\oplus 2}$ is the flat rank-two cone connection, and
the point $q_j$ lifts to the complex line
$$
L_j:=\overline{\pi_j^{-1}(q_j)}
\subset \CBbb^3 ,
$$
where
$$
\pi_j:\CBbb^3\setminus\{0\}\longrightarrow
D_j
$$
is the natural projection. Therefore
$$
\mathcal T_{p_j}(A_\infty)
=
(A_{T,j},Z_{T,j}),
\qquad
A_{T,j}\ \text{is flat},\qquad
Z_{T,j}=[L_j].
$$
In particular, the tangent-cone blow-up cycle has multiplicity one. This
proves the second assertion.

It remains to compare the energy densities. By the energy-measure
convergence in the tangent-cone construction and by the normalized density
convention above,
$$
\Theta_{p_j}(A_\infty)
=
\frac{1}{8\pi^3}
\left(
|F_{A_{T,j}}|^2\,\dvol
+
8\pi^2 Z_{T,j}
\right)(B_1(0)),
$$
where $Z_{T,j}$ is viewed as its associated integration current. Since
$A_{T,j}$ is flat and $Z_{T,j}=[L_j]$, this gives
$$
\Theta_{p_j}(A_\infty)
=
\frac{8\pi^2[L_j](B_1(0))}{8\pi^3}
=
\frac{8\pi^2\cdot \pi}{8\pi^3}
=
1 .
$$

Finally, since the codimension-two bubbling cycle of the original sequence
vanishes, namely $Z_b=0$, we know
$$
\mu_i\rightharpoonup \mu_\infty
$$
as Radon measures on $\PBbb^3$. Thus for $0<r<<1$
$$
\lim_{i\to\infty}\mu_i(B_r(p_j))
=
\mu_\infty(B_r(p_j))
$$
since there is no measure concentration on the boundary. Therefore
$$
\liminf_{i\to\infty}
\frac{\mu_i(B_r(p_j))}{8\pi^3r^2}
=
\frac{\mu_\infty(B_r(p_j))}{8\pi^3r^2}
$$
for such $r$. Letting $r\downarrow 0$, we obtain
$$
\Theta_{p_j}(\{A_i\}_i)
=
\Theta_{p_j}(A_\infty).
$$
Hence, for $j=1,2$,
$$
\Theta_{p_j}(\{A_i\}_i)
=
\Theta_{p_j}(A_\infty)
=
1 .
$$
\end{proof}

In particular, we obtain the following local consequence.

\begin{cor}\label{CorLocalChenSunExample}
Let $B\subset \CBbb^3$ be a sufficiently small ball centered at the origin,
and let $\Fcal_{\min}$ be the reflexive sheaf on $B$ defined by
$$
0\longrightarrow \Ocal_B
\xrightarrow{
\begin{pmatrix}
z_1\\
z_2\\
z_3^2
\end{pmatrix}}
\Ocal_B^3
\longrightarrow \Fcal_{\min}
\longrightarrow 0 .
$$
Then $\Fcal_{\min}$ supports an admissible HYM connection $A_\infty$, smooth on
$B\setminus\{0\}$, which arises as a Uhlenbeck limit of smooth HYM
connections over $B$. 
\end{cor}

\begin{rmk}
The point of this example is that the connection part of the analytic
tangent cone does not encode the original isolated singularity. The sheaf
$\Fcal_{\min}$ has an essential singularity at the origin, whereas the connection
part $A_T$ of the analytic tangent cone is flat. The only remaining record
of curvature concentration is the codimension-two  blow-up cycle $[L]$.
Thus passing to the analytic tangent cone loses the local analytic type of
the original point singularity; it retains only the flat double dual and the
length-one quotient producing the line $L$.
The global construction above shows that this loss of information occurs for
genuine Uhlenbeck limits of smooth HYM connections.
In this sense, the phenomenon is intrinsic to the Uhlenbeck
compactification problem in complex dimension three.
\end{rmk}

\subsection{Bubble analysis}
We now analyze the local bubbling models associated with the global
degeneration near the two singular points $p_1$ and $p_2$. The first step is
to identify the local analytic type of the family.

\begin{lem}\label{LocalModelNearSingularPoints}
For each $i=1,2$, after shrinking the base disc and choosing a sufficiently
small analytic neighborhood $U_i\subset \PBbb^3$ of $p_i$, the family
$\Ecal|_{U_i\times\Delta}$ is locally analytically isomorphic, as a flat
family over $\Delta$, to the minimal model in Example
\ref{SimplestFertileFamily}.
\end{lem}

\begin{proof}
All locally free sheaves appearing in the defining complex of $\Ecal$ are
trivial on $U_i\times\Delta$. Choose trivializations compatible with the
maps in the complex, and choose holomorphic coordinates centered at
$(p_i,0)$, with $t$ the coordinate on the base disc. By the local forms
of the defining sections at $p_i$, the resulting matrix is equivalent, up to
invertible changes of source and target and a holomorphic change of
coordinates, to the defining matrix in Example
\ref{SimplestFertileFamily}. The claim follows.
\end{proof}

Fix one of the two singular points, say $p_i$. Since the two points have
the same local analytic model, it is enough to work at this point. Choose
holomorphic coordinates identifying a neighborhood of $(p_i,0)$ in
$\PBbb^3\times\Delta$ with
$$
B\times\Delta\subset \CBbb^3\times\CBbb,
$$
so that $p_i$ corresponds to $o\in B$ and the central fiber is
$B\times\{0\}$. We keep the notation $\Ecal$ for the restricted family on
$B\times\Delta$. Let
$$
\pi:\widehat{B\times\Delta}\longrightarrow B\times\Delta
$$
be the blow-up of the origin, and denote by
$$
D\cong \PBbb^3
$$
the exceptional divisor. The proper transform of the central fiber meets
$D$ in a hyperplane
$$
H\cong \PBbb^2\subset D .
$$

A reflexive sheaf $\widehat{\Ecal}$ on $\widehat{B\times\Delta}$ is called
an extension of $\Ecal$ at $p_i$ if
$$
\widehat{\Ecal}|_{\widehat{B\times\Delta}\setminus D}
\cong
\pi^*\!\left(\Ecal|_{(B\times\Delta)\setminus\{0\}}\right)
$$
under the natural identification
$$
\widehat{B\times\Delta}\setminus D
\cong
(B\times\Delta)\setminus\{0\}.
$$
Given such an extension, we call $\widehat{\Ecal}|_D$
the algebraic bubble of $\Ecal$ at $p_i$, which is a torsion-free sheaf on $D$ (see \cite[Section~2.2]{ChenSun:20a}).

\begin{lem}
There exists an extension $\widehat{\Ecal}$ of $\Ecal$ at $p_i$ such that
\begin{enumerate}
\item the algebraic bubble $\widehat{\Ecal}|_D$ is a semistable rank-two
bundle on $D\cong \PBbb^3$ with $c_1=0$ and $c_2=1$;

\item the restriction $\widehat{\Ecal}|_H$ is semistable on
$H\cong \PBbb^2$. In addition, its Harder--Narasimhan--Seshadri graded sheaf
recovers the analytic tangent-cone data of $A_\infty$ at $p_i$.
\end{enumerate}

Furthermore, suppose $\widehat{\Ecal}'$ is another extension of $\Ecal$ at $p_i$ so that $\widehat{\Ecal}'|_D$ is semistable, then up to isomorphisms, it is given by tensoring $\widehat{\Ecal}$ with a power of the line bundle
associated with the exceptional divisor.
\end{lem}

\begin{proof}
By Lemma \ref{LocalModelNearSingularPoints}, we may replace the germ of
the family at $(p_i,0)$ by the minimal local model in Example
\ref{SimplestFertileFamily}. In that model the defining presentation is
homogeneous with respect to the variables on $B\times\Delta$. Blowing up
the origin and removing the common exceptional factor in the pulled-back
presentation gives a reflexive extension $\widehat{\Ecal}$ across the
exceptional divisor $D$.

The restriction $\widehat{\Ecal}|_D$ is precisely the link bundle of the
homogeneous model. By the computation in Example
\ref{SimplestFertileFamily}, this is a semistable rank-two bundle on
$D\cong \PBbb^3$ with
$$
c_1(\widehat{\Ecal}|_D)=0,
\qquad
c_2(\widehat{\Ecal}|_D)=1 .
$$
This proves the existence of an extension with the required algebraic
bubble.

It remains to relate its restriction to the central-fiber hyperplane
$H\cong \PBbb^2$ to the analytic tangent cone of $A_\infty$ at $p_i$.
Restricting the homogeneous presentation to $H$ gives a semistable
rank-two sheaf whose Harder--Narasimhan--Seshadri graded object is, after
the normalization $c_1=0$,
$$
\operatorname{Gr}^{\mathrm{HNS}}(\widehat{\Ecal}|_H)
\cong
\Ocal_H\oplus \mathcal I_q
$$
for a point $q\in H$. Equivalently,
$$
\left(
\operatorname{Gr}^{\mathrm{HNS}}(\widehat{\Ecal}|_H)
\right)^{**}
\cong
\Ocal_H^{\oplus 2},
\qquad
\left(
\operatorname{Gr}^{\mathrm{HNS}}(\widehat{\Ecal}|_H)
\right)^{**}
/
\operatorname{Gr}^{\mathrm{HNS}}(\widehat{\Ecal}|_H)
\cong
\Ocal_q .
$$
By the algebro-geometric description of analytic tangent cones in
\cite{ChenSun:18}, this data determines the analytic tangent-cone data of
$A_\infty$ at $p_i$: the connection part is the flat rank-two cone
connection, and the blow-up cycle is the complex line over $q$ with
multiplicity one.

Finally, suppose $\widehat{\Ecal}'$ is another extension of $\Ecal$ at
$p_i$ whose algebraic bubble is semistable. Then
$\widehat{\Ecal}'|_D$ is an optimal algebraic tangent cone. Since the
Harder--Narasimhan filtration has only one step in the semistable case,
the uniqueness theorem for optimal extensions in \cite{ChenSun:20b}
implies that $\widehat{\Ecal}'$ is equivalent to $\widehat{\Ecal}$, namely
$$
\widehat{\Ecal}'\cong \widehat{\Ecal}\otimes \Ocal_{\widehat{B\times\Delta}}(kD)
$$
for some $k\in\mathbb Z$. This proves the
claim.
\end{proof}

We now rescale the original sequence at one of the point singularities.
Fix one of the two points and denote it by $p$. Choose holomorphic
coordinates identifying a neighborhood of $p$ with a ball
$B\subset\CBbb^3$, sending $p$ to the origin. For a sequence of centered
scales $r_i\downarrow 0$, let
$$
\delta_i:B_{r_i^{-1}}\longrightarrow B,\qquad
z\longmapsto r_i z,
$$
and set
$$
\widetilde A_i:=\delta_i^*A_i .
$$
After passing to a subsequence, Uhlenbeck compactness on compact subsets of
$\CBbb^3$ gives an Uhlenbeck limit
$$
(B_\infty,\mathcal B_\infty,Z_b')
$$
of the rescaled sequence $\widetilde A_i$. Here $B_\infty$ is the limiting
admissible HYM connection, $\mathcal B_\infty$ is the limiting reflexive
sheaf on $\CBbb^3$, and $Z_b'$ is the codimension-two blow-up cycle of this
Uhlenbeck limit.

There are two different limiting regimes. If one rescales the admissible
limit $A_\infty$ at $p$, one obtains the analytic tangent-cone data
$$
\mathcal T_p(A_\infty)=(A_{T,p},[L_p]).
$$
In this subsection we instead rescale the original smooth sequence
$A_i$ at a centered scale where the point singularity is resolved. The
resulting limit is a non-flat analytic bubble $B_\infty$. The key point is
that this Uhlenbeck limit has zero blow-up cycle, even though its tangent
cone at infinity has a nontrivial blow-up cycle.

\begin{prop}\label{MinimalAnalyticBubble}
There exists a sequence of centered scales $r_i\downarrow 0$ such that the
rescaled sequence $\widetilde A_i=\delta_i^*A_i$ has an Uhlenbeck limit
$$
(B_\infty,\mathcal B_\infty,Z_b')
$$
on $\CBbb^3$ satisfying
$$
Z_b'=0,
$$
and $\mathcal B_\infty$ is locally free. In particular, up to gauge transforms,
$$
\widetilde A_i\longrightarrow B_\infty
$$
smoothly on compact subsets of $\CBbb^3$, and $B_\infty$ is non-flat.
Moreover, the following hold.
\begin{enumerate}
\item The connection $B_\infty$ has normalized density one at infinity:
$$
\Theta_\infty(B_\infty)
:=
\lim_{R\to\infty}
\frac{1}{8\pi^3R^2}
\int_{B_R(0)} |F_{B_\infty}|^2\,\dvol
=
1 .
$$

\item The analytic tangent-cone data of $B_\infty$ at infinity are
$$
\mathcal T_\infty(B_\infty)=(B_T,[L]),
$$
where $B_T$ is flat and $L\subset\CBbb^3$ is a complex line. In particular,
the tangent-cone blow-up cycle at infinity is a line with multiplicity one.
\end{enumerate}
\end{prop}

\begin{proof}
Choose the centered scales $r_i$ at the resolving scale of the point
singularity. Equivalently, after passing to a subsequence, we may normalize
the rescaled sequence so that a fixed positive amount of curvature remains
on a compact ball. Thus the Uhlenbeck limit of $\widetilde A_i$ is
non-flat (\cite{Nakajima:88}). Let $(B_\infty,\mathcal B_\infty,Z_b')$
be this Uhlenbeck limit. Since the normalized density of the original
sequence at $p$ is
$$
\Theta_p(\{A_i\}_i)=1,
$$
the total normalized density available to any centered rescaling at $p$ is
at most one. Since $B_\infty$ is non-flat, it has positive density at
infinity. On the other hand, any nonzero codimension-two blow-up cycle in
the Uhlenbeck limit of the rescaled sequence contributes at least one unit
of normalized density. Therefore
$$
Z_b'=0.
$$
The same density-counting argument rules out essential singularities of
$\mathcal B_\infty$: an essential singularity would have an analytic tangent
cone with positive integral blow-up contribution, again using at least one
unit of normalized density. Hence $\mathcal B_\infty$ is locally free, and
the rescaled convergence is smooth on compact subsets of $\CBbb^3$.

We now compute the density at infinity. By Price monotonicity, the limit
$$
\Theta_\infty(B_\infty)
=
\lim_{R\to\infty}
\frac{1}{8\pi^3R^2}
\int_{B_R(0)} |F_{B_\infty}|^2\,\dvol
$$
exists. Since $B_\infty$ is non-flat,
$$
\Theta_\infty(B_\infty)>0.
$$
The density bound coming from
$\Theta_p(\{A_i\}_i)=1$ gives
$$
\Theta_\infty(B_\infty)\leq 1.
$$

Take an analytic tangent cone of $B_\infty$ at infinity:
$$
\mathcal T_\infty(B_\infty)=(B_T,Z_T).
$$
Its energy measure has the form
$$
|F_{B_T}|^2\,\dvol+8\pi^2[Z_T].
$$
If $0<\Theta_\infty(B_\infty)<1$, then $Z_T$ must vanish, since any nonzero
conical codimension-two cycle contributes at least one unit of normalized
density. The cone connection $B_T$ would then descend to a HYM connection
on a rank-two polystable sheaf over $\PBbb^2$ with $c_1=0$ and
$c_2<1$, hence with $c_2=0$, forcing the cone to be flat. This contradicts
the positivity of the density. Therefore
$$
\Theta_\infty(B_\infty)=1.
$$

It remains to identify the tangent-cone data at infinity. If the connection
part $B_T$ were non-flat, then, under the determinant normalization, it
would descend to a non-flat HYM connection on a rank-two polystable sheaf
over $\PBbb^2$ with $c_1=0$ and $c_2=1$. Such a connection does not exist
(\cite[Example~4.1.3]{DonaldsonKronheimer:90}). Hence $B_T$ is flat.

Thus all the density at infinity is carried by the blow-up cycle $Z_T$. Since
the normalized density is one, the projectivization of $Z_T$ is an effective
zero-dimensional cycle of degree one on $\PBbb^2$. Therefore it is a single
point with multiplicity one. Equivalently, the corresponding conical cycle in
$\CBbb^3$ is a complex line with multiplicity one:
$$
Z_T=[L].
$$
Hence
$$
\mathcal T_\infty(B_\infty)=(B_T,[L]),
$$
where $B_T$ is flat. This proves the proposition.
\end{proof}
This finishes the proof of Theorem \ref{thm:minimal-global-degeneration}.
\subsection{Compactifications of the minimal analytic bubble}
\label{subsec:compactification-minimal-bubble}
The bubble $B_\infty$ constructed above has a special feature: its density
at infinity exhausts the full normalized density available at the point $p$,
namely
$$
\Theta_\infty(B_\infty)
=
\Theta_p(\{A_i\}_i)
=
\Theta_p(A_\infty)
=
1 .
$$
However, $B_\infty$ is not the analytic tangent cone of the admissible limit
$A_\infty$ at $p$. The word ``minimal'' refers to the density:
a nonzero conical blow-up cycle contributes at least one unit of density,
and $B_\infty$ has exactly one unit at infinity. We call such a bubble a
\emph{minimal-density analytic bubble}, or simply a \emph{minimal analytic
bubble.}
Motivated by the ASD compactification picture on $\CBbb^2$, we use the
following expected compatibility as a conditional hypothesis. Let
$$
H_\infty:=\PBbb^3\setminus \CBbb^3\cong \PBbb^2
$$
be the hyperplane at infinity. We say that the minimal analytic bubble
$B_\infty$ is algebraic in this sense if it admits a torsion-free
compactification $\Gcal$ on $\PBbb^3$ such that
$\Gcal|_{\CBbb^3}$ is the holomorphic bundle defined by $B_\infty$, and the
boundary sheaf
$$
\Gcal_\infty:=\Gcal|_{H_\infty}
$$
is semistable and compatible with the analytic tangent-cone data
$$
\mathcal T_\infty(B_\infty)=(B_T,[L]).
$$
In the present case, this compatibility means that, for the point
$q\in H_\infty$ corresponding to the line $L\subset\CBbb^3$,
$$
\left(\operatorname{Gr}^{\mathrm{HNS}}(\Gcal_\infty)\right)^{**}
\cong
\Ocal_{H_\infty}^{\oplus 2},
\qquad
\left(\operatorname{Gr}^{\mathrm{HNS}}(\Gcal_\infty)\right)^{**}
/
\operatorname{Gr}^{\mathrm{HNS}}(\Gcal_\infty)
\cong
\Ocal_q .
$$
Equivalently, the HYM cone determined by
$\left(\operatorname{Gr}^{\mathrm{HNS}}(\Gcal_\infty)\right)^{**}$ is the
flat cone connection $B_T$, and the associated algebraic blow-up cycle is
the point $q$ with multiplicity one.

Assume that $B_\infty$ is algebraic in the sense above, and fix a compatible
compactification $\Gcal$. Put
$$
\Gcal_\infty:=\left.\Gcal\right|_{H_\infty}.
$$

\begin{prop}\label{Prop-4.10}
Under the compatibility condition above, the boundary sheaf $\Gcal_\infty$
has one of the following two forms. Either it is the unique non-split
extension
$$
0\longrightarrow \Ocal_{H_\infty}
\longrightarrow \Gcal_\infty
\longrightarrow \Ical_q
\longrightarrow 0,
$$
or
$$
\Gcal_\infty\cong \Ocal_{H_\infty}\oplus \Ical_q .
$$
In the first case, $\Gcal_\infty$ is locally free, strictly semistable, and
has $c_1=0$ and $c_2=1$; it is the boundary restriction of the semistable
algebraic bubble constructed above. In the second case, no family forming
the minimal singularity can have an algebraic bubble whose boundary
restriction is $\Gcal_\infty$.
\end{prop}

\begin{proof}
The compatibility condition gives
$$
\left(\operatorname{Gr}^{\mathrm{HNS}}(\Gcal_\infty)\right)^{**}
\cong
\Ocal_{H_\infty}^{\oplus 2},
\qquad
\left(\operatorname{Gr}^{\mathrm{HNS}}(\Gcal_\infty)\right)^{**}
/
\operatorname{Gr}^{\mathrm{HNS}}(\Gcal_\infty)
\cong
\Ocal_q .
$$
In particular,
$$
c_1(\Gcal_\infty)=0,\qquad c_2(\Gcal_\infty)=1 .
$$
Since $\Gcal_\infty$ is semistable on $H_\infty\cong\PBbb^2$, Riemann--Roch
gives $\chi(\Gcal_\infty)=1$, and semistability gives
$H^2(H_\infty,\Gcal_\infty)=0$. Hence
$H^0(H_\infty,\Gcal_\infty)\neq 0$.

Let $s$ be a nonzero section. Since $\Gcal_\infty$ is semistable of slope
zero, the saturation of the subsheaf generated by $s$ is
$\Ocal_{H_\infty}$. Thus we have an exact sequence
$$
0\longrightarrow \Ocal_{H_\infty}
\longrightarrow \Gcal_\infty
\longrightarrow \Ical_Z
\longrightarrow 0
$$
for a zero-dimensional subscheme $Z\subset H_\infty$. The equality
$c_2(\Gcal_\infty)=1$ gives $l(Z)=1$. By the compatibility condition,
$Z$ is the point $q$ corresponding to the line $L\subset\CBbb^3$. Therefore
$$
0\longrightarrow \Ocal_{H_\infty}
\longrightarrow \Gcal_\infty
\longrightarrow \Ical_q
\longrightarrow 0 .
$$
Since
$$
\operatorname{Ext}^1(\Ical_q,\Ocal_{H_\infty})\cong \CBbb,
$$
there are exactly two possibilities: the extension class is zero, giving
$$
\Gcal_\infty\cong \Ocal_{H_\infty}\oplus \Ical_q,
$$
or the extension class is nonzero. In the nonzero case, the extension is
locally free and is the unique strictly semistable rank-two bundle on
$H_\infty$ with $c_1=0$ and $c_2=1$ realizing the above graded object.

It remains to exclude the split case as the boundary restriction of an
algebraic bubble of a family forming the minimal singularity. Suppose, to
the contrary, that such a family $\Ecal'$ exists, and let $\widehat{\Ecal'}$
be an extension of $\Ecal'$ across the exceptional divisor whose algebraic
bubble has boundary restriction
$$
\Ocal_{H_\infty}\oplus \Ical_q .
$$
Restricting $\widehat{\Ecal'}$ to the proper transform of the central fiber
gives a semistable algebraic tangent-cone extension for the minimal
reflexive singularity
$$
0\longrightarrow \Ocal_B
\xrightarrow{
\begin{pmatrix}
z_1\\
z_2\\
z_3^2
\end{pmatrix}}
\Ocal_B^{\oplus 3}
\longrightarrow \Fcal_{\min}
\longrightarrow 0 .
$$
By the computation of the semistable extension for this local model above,
the restriction to the exceptional hyperplane is the non-split extension
$$
0\longrightarrow \Ocal_{H_\infty}
\longrightarrow \Gcal_\infty^{\mathrm{ss}}
\longrightarrow \Ical_q
\longrightarrow 0,
$$
and in particular is locally free. This contradicts the split boundary
restriction
$$
\Ocal_{H_\infty}\oplus \Ical_q,
$$
which is not locally free at $q$. Hence the split case cannot arise from an
algebraic bubble of any family forming the minimal singularity.
\end{proof}

\begin{rmk}
A folklore expectation for the analytic bubble itself is
that the split boundary type should occur at infinity. If this expectation is
correct, then the proposition indicates a genuine distinction between
analytic bubbles and algebraic bubbles arising from algebraic degenerations.
Determining the precise behavior in this case remains an interesting
question.
\end{rmk}
   
\section{Smoothable elementary modifications of projective cones
with explicit algebraic bubbling}\label{GlobalExamples}
In this section, we construct smoothings of finite-colength elementary
modifications of reflexive projective cones on $\PBbb^3$. Starting from a
stable rank-two bundle on $\PBbb^2$, we form the associated reflexive cone
$\Fcal$ on $\PBbb^3$ and construct a flat family whose central fiber
$\Ccal$ is an elementary modification of $\Fcal$ at the vertex, while the general fibers are
locally free and stable. These examples give an explicit global setting in
which the algebraic bubble restricts, along the hyperplane at infinity, to
the link bundle of the analytic tangent cone. They also show that the
algebraic multiplicity defined above and the analytic energy density measure
different features of the degeneration, and need not be comparable in
general.

Let $H_\infty:=\{Z_4=0\}\cong \PBbb^2$, so that
$\PBbb^3=\CBbb^3\cup H_\infty$. Endow $\PBbb^3$ with the
Fubini--Study metric $\omega_{FS}$, and endow $H_\infty$ with the
induced Fubini--Study metric $\underline{\omega}_{FS}$. We write a point
of $\PBbb^3$ as $Z=[Z_1,Z_2,Z_3,Z_4]$, and set
$v=[0,0,0,1]$. Let
$$
\pi:\PBbb^3\setminus\{v\}\longrightarrow H_\infty,\qquad
[Z_1,Z_2,Z_3,Z_4]\longmapsto [Z_1,Z_2,Z_3,0],
$$
be the projection from the vertex. Given a rank-two stable bundle
$\underline{\Ecal}$ on $H_\infty\cong \PBbb^2$ with
$c_1(\underline{\Ecal})=0$, define the associated reflexive cone by
$$
\Fcal:=j_*\pi^*\underline{\Ecal},\qquad
j:\PBbb^3\setminus\{v\}\hookrightarrow \PBbb^3 .
$$
Let $\underline{H}$ be the HYM metric on $\underline{\Ecal}$ with respect
to $\underline{\omega}_{FS}$, and let $H=\pi^*\underline{H}$ be the
induced metric on the locally free locus of $\Fcal$. We first record the
following observation.

\begin{lem}\label{Lemma5.1}
The metric $H$ is an admissible HYM metric on $\Fcal$ with respect to
$\omega_{FS}$. Moreover, on $\CBbb^3\setminus\{v\}$ the same metric is HYM
with respect to the Euclidean metric $\omega_0$.
\end{lem}

\begin{proof}
Since $\underline{H}$ is HYM and $c_1(\underline{\Ecal})=0$, we have
$$
\sqrt{-1}F_{\underline{H}}\wedge \underline{\omega}_{FS}=0.
$$
On $\CBbb^3\setminus\{v\}$, write
$$
\alpha:=\pi^*\underline{\omega}_{FS},
\qquad
\eta:=2r\,dr\wedge\theta ,
$$
where $\theta$ is the standard contact form. Then
$$
\omega_0=r^2\alpha+\eta
$$
and
$$
\omega_{FS}
=
\frac{r^2}{1+r^2}\alpha
+
\frac{1}{(1+r^2)^2}\eta .
$$
The curvature of $H$ is $F_H=\pi^*F_{\underline{H}}$. It is horizontal, so
$F_H\wedge\alpha^2=0$, while $\eta^2=0$. Hence
$$
\sqrt{-1}F_H\wedge\omega_0^2
=
2r^2\,\sqrt{-1}F_H\wedge\alpha\wedge\eta
$$
and
$$
\sqrt{-1}F_H\wedge\omega_{FS}^2
=
\frac{2r^2}{(1+r^2)^3}
\sqrt{-1}F_H\wedge\alpha\wedge\eta .
$$
But
$$
\sqrt{-1}F_H\wedge\alpha
=
\sqrt{-1}\pi^*
\bigl(F_{\underline{H}}\wedge\underline{\omega}_{FS}\bigr)
=
0.
$$
Thus $H$ is HYM with respect to $\omega_0$ on
$\CBbb^3\setminus\{v\}$, and it is HYM with respect to $\omega_{FS}$ on
$\CBbb^3\setminus\{v\}$. The latter equation extends across
$H_\infty$ by smoothness, hence holds on $\PBbb^3\setminus\{v\}$.

It remains only to check admissibility near $v$. Near $v$, the metrics
$\omega_{FS}$ and $\omega_0$ are uniformly equivalent, and since
$F_H=\pi^*F_{\underline{H}}$ is horizontal, we have
$$
|F_H|_{\omega_0}\leq C r^{-2}.
$$
Therefore
$$
\int_{B_\epsilon(v)\setminus\{v\}}
|F_H|_{\omega_{FS}}^2\,dV_{\omega_{FS}}
\leq
C\int_0^\epsilon r^{-4}r^5\,dr
<\infty .
$$
\end{proof}
On the algebraic side, we have the following.

\begin{prop}\label{smoothProjectiveCones}
Suppose $c_1(\underline{\Ecal})=0$. Then there exists a family $\Ecal$ over
$\PBbb^3\times \CBbb$ such that
\begin{enumerate}
\item the central fiber $\Ccal$ is a torsion free sheaf with a point singularity at the vertex,
and $\Ecal_t$ is locally free and stable for every $t\neq 0$. Furthermore,
$\Ccal^{**}\cong \Fcal$, and $\Fcal/\Ccal$ is nonzero and supported at the
vertex. In addition,
$$
2l(\Fcal/\Ccal)= c_3(\Ccal^{**})
= l(\Ecal xt^1(\Ccal,\Ocal_B)).
$$

\item There exists an extension $\widehat{\Ecal}$ of $\Ecal$ at the vertex such that the
algebraic bubble $\Gcal$ is a rank-two stable bundle over $\PBbb^3$ and
satisfies
$$
\Gcal|_{H_\infty}\cong \underline{\Ecal}.
$$
Moreover, up to tensoring with powers of the exceptional divisor, any other extensions with semi-stable algebraic bubbles is isomorphic to $\widehat{\Ecal}$.
\end{enumerate}
\end{prop}

The proof has two inputs. First, Donaldson's interpretation of the ADHM
construction gives a stable rank-two bundle on $\PBbb^3$ whose restriction
to $H_\infty\cong\PBbb^2$ is the prescribed bundle
$\underline{\Ecal}$. Second, we use an explicit resolution of a natural
rational map on $\PBbb^3\times\CBbb$ to produce the desired degeneration.
We begin with the extension step.

\begin{lem}
There exists a stable rank-two bundle $\Gcal$ on $\PBbb^3$ such that
$$
\Gcal|_{H_\infty}\cong \underline{\Ecal}.
$$
\end{lem}

\begin{proof}
By the Grauert--Mülich theorem (\cite{OSS11}), since $\underline{\Ecal}$ is stable of
rank two on $H_\infty\cong \PBbb^2$ and
$c_1(\underline{\Ecal})=0$, its restriction to a general line
$l\subset H_\infty$ is trivial:
$$
\underline{\Ecal}|_{l}\cong \Ocal_{l}^{\oplus 2}.
$$
Choosing such a line and a trivialization of
$\underline{\Ecal}|_{l}$, Donaldson's ADHM construction
\cite{Donaldson:84} gives an instanton bundle $\Gcal$ on $\PBbb^3$
whose restriction to $H_\infty$ is isomorphic to
$\underline{\Ecal}$. In particular, $\Gcal$ is a rank-two stable bundle
with $c_1(\Gcal)=0$ and
$$
H^1(\PBbb^3,\Gcal(-2))=0 .
$$
\end{proof}
We now construct the birational model underlying the degeneration. For
$t\in\Delta^*$, let
$$
\Phi_t:\PBbb^3\longrightarrow \PBbb^3,\qquad
[Z_1,Z_2,Z_3,Z_4]\longmapsto [Z_1,Z_2,Z_3,tZ_4].
$$
This is a one-parameter family of automorphisms fixing $H_\infty$
pointwise. As $t\to 0$, it collapses each line through
$v=[0,0,0,1]$ to its intersection with $H_\infty$; hence the limiting map
is the projection from $v$, with indeterminacy only at $v$. Equivalently,
we consider the rational map
$$
\Phi:\PBbb^3\times\Delta\dashedrightarrow\PBbb^3,\qquad
([Z_1,Z_2,Z_3,Z_4],t)\longmapsto [Z_1,Z_2,Z_3,tZ_4].
$$
Its graph closure is the blow-up of $\PBbb^3\times\Delta$ at $(v,0)$.
We denote it by
$$
p:\widehat{\PBbb^3\times\Delta}\longrightarrow \PBbb^3\times\Delta .
$$
In coordinates, it is the subvariety of
$\PBbb^3\times\Delta\times\PBbb^3$ given by
$$
\widehat{\PBbb^3\times\Delta}
=
\left\{
([Z],t,[L]):
\operatorname{rank}
\begin{pmatrix}
L_1&L_2&L_3&T\\
Z_1&Z_2&Z_3&tZ_4
\end{pmatrix}
\leq 1
\right\},
$$
where $[L]=[L_1,L_2,L_3,T]$, and $p$ is the projection to the first two
factors. The central fiber of $p$ has two components,
$$
p^{-1}(0)=\operatorname{Bl}_v\PBbb^3\cup_E \PBbb^3_{\mathrm{exc}},
\qquad E\cong \PBbb^2.
$$
The resolved map restricts to the projection
$\operatorname{Bl}_v\PBbb^3\to H_\infty$ on the strict transform, and to an
isomorphism $\PBbb^3_{\mathrm{exc}}\to\PBbb^3$ on the exceptional
component.

Let
$$
\widehat{\Phi}:\widehat{\PBbb^3\times\Delta}\longrightarrow \PBbb^3
$$
be the projection to the last factor in the graph-closure description.
Then $\widehat{\Phi}$ resolves $\Phi$, that is,
$\widehat{\Phi}=\Phi\circ p$ over the locus where $\Phi$ is defined.
On the central fiber, $\widehat{\Phi}$ restricts to the projection
$$
\operatorname{Bl}_v\PBbb^3\longrightarrow H_\infty
$$
on the strict transform, and to an isomorphism
$$
\PBbb^3_{\mathrm{exc}}\xrightarrow{\ \sim\ }\PBbb^3
$$
on the exceptional component. This allows to identify $\PBbb^3_{\mathrm{exc}}$ with $\PBbb^3$ below. 

Now let $\widehat{\Ecal}=\widehat{\Phi}^*\Gcal$ and
$\Ecal=p_*\widehat{\Ecal}$. As in \cite[Lemma $2.7$]{Chen2025}, we first note the following.

\begin{lem}\label{Lemma-5.4}
The sheaf $\Ecal=p_*\widehat{\Ecal}$ is reflexive on
$\PBbb^3\times\Delta$ with an isolated point singularity at $(v,0)$. In particular, $\Ecal$ is flat over $\Delta$.
\end{lem}

\begin{proof}
The morphism $p$ is the blow-up of the smooth fourfold
$\PBbb^3\times\Delta$ at the point $(v,0)$, and $\widehat{\Ecal}$ is
locally free. The same argument as in \cite[Lemma $2.7$]{Chen2025} shows
that $p_*\widehat{\Ecal}$ is reflexive. Since the base is the smooth curve
$\Delta$, reflexivity also implies flatness over $\Delta$.
\end{proof}

For $t\neq 0$, the morphism $p$ is an isomorphism over
$\PBbb^3\times\{t\}$. Hence
$$
\Ecal_t\cong \Phi_t^*\Gcal .
$$
In particular, $\Ecal_t$ is locally free and stable for every $t\neq 0$.

\begin{lem}
$\left.\widehat{\Ecal}\right|_{\PBbb^3_{\mathrm{exc}}}\cong \Gcal.$
\end{lem}

\begin{proof}
By Lemma \ref{Lemma-5.4}, $\Ecal=p_*\widehat{\Ecal}$ is reflexive, so
$\widehat{\Ecal}$ is an extension of $\Ecal$ at $(v,0)$. On the exceptional component, the resolved map
$\widehat{\Phi}$ restricts to an isomorphism
$$
\PBbb^3_{\mathrm{exc}}\xrightarrow{\ \sim\ }\PBbb^3 .
$$
Therefore
$$
\left.\widehat{\Ecal}\right|_{\PBbb^3_{\mathrm{exc}}}
=
\left.(\widehat{\Phi}^*\Gcal)\right|_{\PBbb^3_{\mathrm{exc}}}
\cong \Gcal .
$$
\end{proof}

\begin{lem}
The central fiber $\Ccal$ is torsion free with an isolated point singularity at the vertex and satisfies $\Ccal^{**}\cong \Fcal.$
Moreover, $\Ccal^{**}/\Ccal$ has finite length and is supported at the vertex.
\end{lem}

\begin{proof}
Since $\Ecal$ is reflexive with an isolated point singularity at $(v,0)$, by \cite[Lemma $3.14$]{ChenSun:20a},  $\Ccal$ is torsion free with an isolated point singularity at the vertex. By construction, it is then clear that
$$
\Ccal^{**}\cong \Fcal .
$$
Therefore $\Ccal^{**}/\Ccal$ is supported at $v$, and hence has finite
length.
\end{proof}

\begin{lem}
The central fiber satisfies
$$
c_3(\Ccal)=0.
$$
In particular,
$$
h^0\bigl(\Ecal xt^1(\Ccal,\Ocal_{\PBbb^3})\bigr)
=
c_3(\Ccal^{**})
=
2\,l(\Ccal^{**}/\Ccal).
$$
\end{lem}

\begin{proof}
Flatness of $\Ecal$ over $\Delta$ gives
$$
c_3(\Ccal)=c_3(\Ecal_t).
$$
For $t\neq 0$, the sheaf $\Ecal_t$ is a rank-two vector bundle, so
$c_3(\Ecal_t)=0$. Hence
$$
c_3(\Ccal)=0.
$$
Since $\Ccal^{**}/\Ccal$ is supported at the vertex, we know
$$
\Ecal xt^1(\Ccal,\Ocal_{\PBbb^3})
\cong
\Ecal xt^1(\Ccal^{**},\Ocal_{\PBbb^3}).
$$
By \cite[Proposition $2.6$]{Hartshorne:1980}, applied to the rank-two
reflexive sheaf $\Ccal^{**}$,
$$
h^0\bigl(\Ecal xt^1(\Ccal,\Ocal_{\PBbb^3})\bigr)
=
c_3(\Ccal^{**}).
$$
On the other hand, additivity of Chern characters gives
$$
\operatorname{ch}_3(\Ccal^{**})
=
\operatorname{ch}_3(\Ccal)+\operatorname{ch}_3(Q)
=
\operatorname{ch}_3(\Ccal)+l(Q).
$$
Since $c_1(\Ccal)=c_1(\Ccal^{**})=0$,
$$
\operatorname{ch}_3(\Ccal)=\frac{1}{2}c_3(\Ccal),
\qquad
\operatorname{ch}_3(\Ccal^{**})=\frac{1}{2}c_3(\Ccal^{**}).
$$
Together with $c_3(\Ccal)=0$, this gives
$$
c_3(\Ccal^{**})
=
2\,l(Q)
=
2\,l(\Ccal^{**}/\Ccal).
$$
The result follows.
\end{proof}
Combining the general-fiber identification above with the preceding lemmas,
the constructed family satisfies the asserted properties in Proposition~\ref{smoothProjectiveCones} except the uniqueness, which follows from the main results in \cite{ChenSun:20b}. This finishes the proof of Proposition~\ref{smoothProjectiveCones}.

We now turn to the analytic side. Fix a sequence $t_i\to 0$ with
$t_i\neq 0$. By the Donaldson--Uhlenbeck--Yau theorem, each stable bundle
$\Ecal_{t_i}$ admits a HYM metric $H_i$ with respect to $\omega_{FS}$.
Let $A_i$ be its Chern connection. After choosing smooth bundle
identifications and a reference Hermitian metric, we regard the $A_i$ as
unitary connections on a fixed Hermitian bundle. By passing to a subsequence, we obtain a Uhlenbeck limit $(A_\infty, \Ecal_\infty, Z_b)$.

\begin{prop}
\begin{enumerate}
\item There is no pure codimension two bubbling. More precisely, $Z_b=0$. In particular, 
$$
|F_{A_i}|_{\omega_{FS}}^2\,\dvol_{\omega_{FS}}
\rightharpoonup
|F_{A_\infty}|_{\omega_{FS}}^2\,\dvol_{\omega_{FS}}
$$
as Radon measures on $\PBbb^3$.

\item the limiting HYM connection is a cone. More precisely, $
A_\infty=\pi^*\underline{A}
$ on $\PBbb^3\setminus\{v\}$, where $\underline{A}$ is the Chern connection
of the HYM metric $\underline{H}$ on $\underline{\Ecal}$. In particular, the algebraic bubble restricts at infinity to the link bundle of
$A_\infty$:
$$
\left.\Gcal\right|_{H_\infty}\cong \underline{\Ecal}.
$$

\item The related energy densities at the vertex satisfy
$$
\Theta_v(A_\infty)=\Delta(\Gcal)
$$
where $\Delta(\Gcal)= c_2(\Gcal)-\frac{1}{4}c_1(\Gcal)^2$.
\end{enumerate}
\end{prop}

\begin{proof}
The central reflexive hull is $\Ccal^{**}\cong \Fcal$, and
$\Ccal^{**}/\Ccal$ is supported at the vertex. Hence
the argument for Proposition~\ref{StableReflexiveSheafImpliesHYMConvergence} applies to
the present family and gives convergence to the admissible HYM connection
on $\Fcal$ with zero bubbling cycle. By the construction of the admissible
metric on $\Fcal$ (see Lemma \ref{Lemma5.1}), this limiting connection is
$A_\infty=\pi^*\underline{A}$. This proves the first two assertions. The density statement is the standard cone computation:
$$
\Theta_v(A_\infty)
=
\frac{1}{8\pi^2}\int_{H_\infty}|F_{\underline{A}}|^2\,\dvol_{\underline{\omega}_{FS}}
=
\,\Delta(\underline{\Ecal})
=
\,\Delta(\Gcal).
$$
The last
equality uses $\left.\Gcal\right|_{H_\infty}\cong\underline{\Ecal}$.
\end{proof}

The following example shows that the algebraic multiplicities above and the
analytic energy density may scale differently.

\begin{exam}\label{NoRelation}
For $k\geq 1$, consider on $H_\infty\cong\PBbb^2$ the exact sequence
$$
0\longrightarrow \Ocal(-3k)
\xrightarrow{\left(\begin{smallmatrix}
Z_1^{2k}\\ Z_2^{2k}\\ Z_3^{2k}
\end{smallmatrix}\right)}
\Ocal(-k)^{\oplus 3}
\longrightarrow \underline{\Ecal}
\longrightarrow 0 .
$$
The three entries have no common zero on $H_\infty$, so
$\underline{\Ecal}$ is a rank-two vector bundle. Moreover,
$H^0(H_\infty,\underline{\Ecal}(-a))=0$ for all $a\geq 0$, hence
$\underline{\Ecal}$ is stable. If $h=c_1(\Ocal(1))$, then
$$
c(\underline{\Ecal})
=
\frac{(1-kh)^3}{1-3kh}
=
1+3k^2h^2 .
$$
Thus $c_1(\underline{\Ecal})=0$ and
$$
c_2(\underline{\Ecal})=3k^2 .
$$

Let $\Fcal=j_*\pi^*\underline{\Ecal}$ be the associated reflexive cone.
It has the presentation
$$
0\longrightarrow \Ocal_{\PBbb^3}(-3k)
\longrightarrow \Ocal_{\PBbb^3}(-k)^{\oplus 3}
\longrightarrow \Fcal
\longrightarrow 0 ,
$$
with the same column map. Hence
$$
c_3(\Fcal)=8k^3 .
$$
For the smoothing constructed above, $\Ccal^{**}\cong\Fcal$, and the
identity
$$
m_0^{\mathrm{alg}}=2m^{\mathrm{alg}}=c_3(\Ccal^{**})
$$
gives
$$
m^{\mathrm{alg}}=4k^3,\qquad m_0^{\mathrm{alg}}=8k^3 .
$$

On the other hand, the vertex density satisfies
$$
\Theta_v(A_\infty)
=
\Delta(\underline{\Ecal})
=
c_2(\underline{\Ecal})
=
3k^2.
$$
In particular, there is no universal
proportionality between the algebraic bubbling multiplicity  and the analytic energy
density at the vertex.
\end{exam}

\bibliography{papers}   
\end{document}